\newif\ifsiam
\siamfalse % enable this if using the standard latex `article` class
\ifsiam

    \documentclass{siamart1116}

    \newcommand{\algcaption}[1]{\caption{#1}}

\else

    \documentclass[12pt]{article}

    \usepackage{latexsym,amsfonts,amsmath,theorem,amssymb}
    \newtheorem{theorem}{Theorem}
    \newtheorem{lemma}[theorem]{Lemma}
    \newtheorem{corollary}[theorem]{Corollary}
    \newtheorem{proposition}[theorem]{Proposition}
    \newtheorem{definition}[theorem]{Definition}
    \newtheorem{algorithm}{Algorithm}

    \newcommand{\algcaption}[1]{#1}

    \newenvironment{proof}{\begin{trivlist}
        \item[\hskip\labelsep{\bf Proof.}]}{$\hfill\Box$\end{trivlist}}
    \newcommand{\funding}[1]{#1}
    \newenvironment{keywords}{\paragraph{Keywords:}}{\par}
    \newenvironment{AMS}{\paragraph{AMS:}}{\par}
    \newcommand{\email}[1]{#1}
    \newcommand{\headers}[2]{}

    \usepackage[hidelinks,colorlinks,allcolors=magenta,hypertexnames=false,hyperindex=true,pdfpagelabels,linktoc=all]{hyperref}

    \usepackage[usenames]{color}
\fi

\usepackage{amsfonts}
\usepackage{graphicx}
\graphicspath{{}{graphs/}}
\usepackage{epstopdf}
\usepackage{algorithmic}
\usepackage{amsopn}

\ifpdf
  \DeclareGraphicsExtensions{.pdf}
\else
  \DeclareGraphicsExtensions{.eps}
\fi

\numberwithin{theorem}{section}

\newcommand{\TheTitle}{Analysis of circulant embedding methods \\ for sampling stationary random fields}
\newcommand{\TheAuthors}{I. G. Graham, F. Y. Kuo, D. Nuyens, R. Scheichl, and I. H. Sloan}

\headers{Analysis of Circulant Embedding}{\TheAuthors}

\title{{\TheTitle}\thanks{\ifsiam Revised manuscript submitted to the editors \today. \else \today. \fi %
\funding{The authors acknowledge financial support from the Australian Research
    Council (FT130100655; DP150101770), the KU Leuven research fund
    (OT:3E130287; C3:3E150478), the Taiwanese National Center for
    Theoretical Sciences' Mathematics Division, and the Statistical
    and Applied Mathematical Sciences Institute (SAMSI) under its 2017
    Program on Quasi-Monte Carlo and High-Dimensional Sampling Methods for
    Applied Mathematics.}}}

\author{
I.G.\ Graham
\thanks{Dept.\ of Mathematical Sciences, University of Bath, Bath BA2
7AY, UK (\email{I.G.Graham@bath.ac.uk}).}
\and
F.Y.\ Kuo
\thanks{School of Mathematics and Statistics, University of NSW,
Sydney NSW 2052, Australia (\email{f.kuo@unsw.edu.au}).}
\and
D.\ Nuyens \thanks{Dept.\ of Computer Science, KU Leuven, Celestijnenlaan
200A, B-3001 Leuven, Belgium (\email{dirk.nuyens@cs.kuleuven.be}).  }
\and
R.\ Scheichl\thanks{Dept.\ of Mathematical Sciences, University of Bath,
Bath BA2 7AY, UK (\email{R.Scheichl@bath.ac.uk}).}
\and
I.H.\ Sloan
\thanks{School of Mathematics and Statistics, University of NSW,
Sydney NSW 2052, Australia (\email{i.sloan@unsw.edu.au}).}
}

\usepackage[a4paper]{geometry}
\usepackage{pdfsync}
\usepackage{latexsym,amsfonts,amsmath,amssymb}

\newtheorem{notation}[theorem]{Notation}
 \newtheorem{discussion}[theorem]{Discussion}
 \newenvironment{proofof}[1]{\begin{trivlist}
   \item[\hskip\labelsep{\bf Proof of {#1}.}]}{$\hfill\Box$\end{trivlist}}
 {\theoremstyle{plain} \theorembodyfont{\rmfamily}
 \newtheorem{remark}[theorem]{Remark}}
 {\theoremstyle{plain} \theorembodyfont{\rmfamily}
 \newtheorem{example}[theorem]{Example}}

\numberwithin{equation}{section}

\newcommand{\oZmd}{\overline{\bbZ}_m^d}
\newcommand{\rhoext}{\rho^{\mathrm{ext}}}
\newcommand{\Rext}{R^{\mathrm{ext}}}

\newcommand{\Bext}{B^{\mathrm{ext}}}
\newcommand{\Qext}{Q^{\mathrm{ext}}}
\newcommand{\Lambdaext}{\Lambda^{\mathrm{ext}}}
\newcommand{\lambdaext}{\lambda^{\mathrm{ext}}}
\newcommand{\cRext}{\mathcal{R}^{\mathrm{ext}}}

\newcommand{\eps}{\varepsilon}

\newcommand{\len}{{\ell}}

\newcommand{\Zbar}{\overline{Z}}
\newcommand{\bZbar}{\overline{\bsZ}}

\newcommand{\bk}{{\boldsymbol{k}}}
\newcommand{\bsk}{{\boldsymbol{k}}}

\newcommand{\br}{{\boldsymbol{r}}} % dn changed from mathbf to boldsymbol

\newcommand{\bszeta}{{\boldsymbol{\zeta}}}

\newcommand{\bsB}{\mathbf{B}}

\newcommand{\cN}{\mathcal{N}}

\newcommand{\Mat}{Mat\'{e}rn\ }

\newcommand{\bkappa}{{\boldsymbol{\kappa}}}

\newcommand{\bsgamma}{{\boldsymbol{\gamma}}}

\newcommand{\bsr}{{\boldsymbol{r}}}

\newcommand{\bsq}{{\boldsymbol{q}}}

\newcommand{\bst}{{\boldsymbol{t}}}

\newcommand{\bsv}{{\boldsymbol{v}}}
\newcommand{\bsV}{{\boldsymbol{V}}}
\newcommand{\bsx}{{\boldsymbol{x}}}

\newcommand{\bsy}{{\boldsymbol{y}}}
\newcommand{\bsY}{\boldsymbol{Y}}
\newcommand{\bsz}{{\boldsymbol{z}}}
\newcommand{\bsZ}{{\boldsymbol{Z}}}
\newcommand{\bsw}{{\boldsymbol{w}}}

\newcommand{\bszero}{{\boldsymbol{0}}}

\newcommand{\bsxi}{\boldsymbol{\xi}}

\newcommand{\rd}{\mathrm{d}}
\newcommand{\ri}{\mathrm{i}}
\newcommand{\bbR}{\mathbb{R}}
\newcommand{\bbZ}{\mathbb{Z}}

\newcommand{\bbE}{\mathbb{E}}

\newcommand{\calF}{\mathcal{F}}

\newcommand{\cF}{\mathcal{F}}

\newcommand{\calO}{\mathcal{O}}

\renewcommand{\Re}{\mathfrak{Re}}
\renewcommand{\Im}{\mathfrak{Im}}

\newcommand{\mask}[1]{}

\newcommand{\cR}{\mathcal{R}}

\newcommand{\wM}{\widetilde{\Psi}}
\newcommand{\wrho}{\widehat{\rho}}

\renewcommand{\forall}{\text{ for all }}

\newcommand{\bR}{\mathbb{R}}

\makeatletter\@addtoreset{equation}{section}\makeatother

\definecolor{darkred}{RGB}{139,0,0}
\definecolor{darkgreen}{RGB}{0,100,0}
\definecolor{darkmagenta}{RGB}{139,0,139}
\definecolor{darkpurple}{RGB}{110,0,180}
\definecolor{darkblue}{RGB}{40,0,200}
\definecolor{darkorange}{RGB}{170,100,30}
\newcommand{\ceil}[1]{\left\lceil#1\right\rceil}
\newcommand{\floor}[1]{\left\lfloor#1\right\rfloor}
\newcommand{\rcov}{r_{\mathrm{cov}}}

\begin{document}

\maketitle

\begin{abstract}
A standard problem in uncertainty quantification and in computational
statistics is the sampling of stationary Gaussian random fields with given
covariance in a $d$-dimensional (physical) domain. In many applications it
is sufficient to perform the sampling on a regular grid on a cube
enclosing the physical domain, in which case the corresponding covariance
matrix is nested block Toeplitz. After extension to a nested block
circulant matrix, this can be diagonalised by FFT -- the ``circulant
embedding method''. Provided the circulant matrix is positive definite,
this provides a finite expansion of the field in terms of a deterministic
basis, with coefficients given by i.i.d. standard normals. In this paper
we prove, under mild conditions, that the positive definiteness of the
circulant matrix is always guaranteed, provided the enclosing cube is
sufficiently large. We examine in detail the case of the Mat\'{e}rn
covariance, and prove (for fixed correlation length) that, as
$h_0\rightarrow 0$, positive definiteness is guaranteed when the random
field is sampled on {a} cube of size order $(1 + \nu^{1/2} \log
h_0^{-1})$ times larger than the size of the physical domain. (Here
$h_0$ is the mesh spacing of the regular grid and $\nu$ the Mat\'{e}rn
smoothness parameter.) We show that the sampling cube can become smaller
as the correlation length decreases when $h_0$ and $\nu$ are fixed. Our
results are confirmed by numerical experiments. We prove several results
about the decay of the eigenvalues of the circulant matrix. These lead
to the conjecture, verified by numerical experiment, that they decay
with the same rate as the Karhunen--Lo\`{e}ve eigenvalues of the
covariance operator. The method analysed here complements the numerical
experiments for uncertainty quantification in porous media problems in
an earlier paper by  the same authors in {\em J. Comp. Physics}. 230
(2011), pp. 3668--3694.
\end{abstract}

% REQUIRED

% REQUIRED

%\vspace{-1ex}

\section{Introduction} \label{sec:Intro}

In recent years there has been a huge growth in interest in uncertainty
quantification (UQ) for physical models involving partial differential
equations. In this context, the forward problem of UQ consists of
describing the statistics of outputs (solutions) of a PDE model, given
statistical assumptions on its inputs (e.g., its coefficients). This has
led to the widespread study of model PDE problems where the coefficients
are given as random fields. Various flavours of solution method (e.g.,
Stochastic Galerkin or various sampling methods) have been formulated
under the assumption that the random coefficient field has a separable
expansion in physical/probability space -- for example a suitably
truncated Karhunen--Lo\'{e}ve (KL) expansion, \cite{GhSp:91},
\cite{LPS14}.

In its standard form, the KL expansion is in principle infinite, and
requires the computation of the (infinitely many) eigenpairs of the
integral operator with kernel given by the covariance of the field. It has
to be truncated to be computable, even when the field is only required at
a finite set of physical spatial points. However if it is known from the
outset that the coefficient field is only required at a finite set of
physical points (as is the case in typical implementations of finite
element methods for solving the PDE), then a different point of view
emerges. The required field is now a random vector; in the Gaussian case
it is characterised by its mean and covariance matrix. A real, for example
Cholesky, factorization of the covariance matrix provides a finite
separable expansion of the random vector, with no need for any truncation.

In the paper \cite{GrKuNuScSl:11}, the authors proposed a practical
algorithm for solving a class of elliptic PDEs
 with coefficients given by statistically homogeneous lognormal
random fields with low regularity.  Lognormal random fields are commonly
used in applications, for example in hydrology (see, e.g.,
\cite{NaHaSu:98a,NaHaSu:98b} and the references there).

The PDE was solved by piecewise linear finite elements on a uniform grid.
The stiffness matrix was obtained by an appropriate quadrature rule, with
the field values at the quadrature points being obtained via a
factorization of the covariance matrix using the circulant embedding
technique described below. The method was found to be effective even for
problems with high stochastic dimension, but \cite{GrKuNuScSl:11} did not
contain a convergence analysis of the algorithm. The main purpose of the
present paper is to provide an analysis for the circulant embedding part
of the algorithm of \cite{GrKuNuScSl:11}. The analysis of the
corresponding uncertainty quantification algorithm for the PDE is done
in~\cite{paper2}.

We consider here the fast evaluation of a Gaussian random field $Z(\bsx,
\omega)$ with prescribed mean $\Zbar(\bsx)$ and covariance
\begin{equation}\label{eq:covar}
  \rcov(\bsx,\bsx') \,:=\,
  \mathbb{E}[(Z(\bsx,\cdot)-\Zbar(\bsx))(Z(\bsx',\cdot)-\Zbar(\bsx')],
\end{equation}
where the expectation is with respect to the Gaussian measure. Throughout
we will assume that $Z$ is \emph{stationary} (see, e.g.,
\cite[p.~24]{Ad:81}), i.e., its covariance function satisfies
\begin{equation} \label{eq:defrho1}
 \rcov(\bsx,\bsx') \,=\, \rho(\bsx-\bsx'),
\end{equation}
for some function $\rho:\mathbb{R}^d\to\mathbb{R}$. Note that we assume
here that $\rho$ is defined on all of $\bR^d$, as it is in many
applications, although strictly speaking  we only need $\rho$ to be
defined on a sufficiently large ball. Note also that \eqref{eq:covar} and
\eqref{eq:defrho1} imply that $\rho$ is symmetric, i.e.,
\begin{equation} \label{eq:symm}
\rho(\bsx) = \rho(-\bsx), \quad \text{for all} \quad \bsx \in \mathbb{R}^d \ ,
\end{equation}
and that $\rho$ is also positive semidefinite (see \S
\ref{subsec:posdef}). Further assumptions on $\rho$ will be given below. A
particular case, to be discussed extensively, is the \Mat covariance
defined in Example~\ref{ex:Mat} below.

We shall consider the problem of evaluating
$Z(\bsx, \omega)$ at a uniform grid of
\[
 M \,=\, (m_0+1)^d
\]
points on the $d$-dimensional unit cube $[0,1]^d$, with integer $m_0$
fixed, and with grid spacing $h_0 := 1/m_0$. (The extension to general
tensor product grids is straightforward and not discussed here.) Denoting
the grid points by $\bsx_1,\bsx_2,\ldots,\bsx_M$, we wish to obtain
samples of the random vector:
\begin{equation*} %\label{eq:Zvec}
 \bsZ(\omega) \,:=\, (Z(\bsx_1,\omega), \ldots , Z(\bsx_M, \omega))^{\top} \ .
\end{equation*}
This is a Gaussian random vector with mean $\overline{\bsZ} :=
(\overline{Z}(\bsx_1),\ldots,\overline{Z}(\bsx_M))^\top$ and a positive
semidefinite covariance matrix
\begin{equation} \label{eq:Rmatrix}
  R \,=\,  [\rho(\bsx_i-\bsx_j)]_{i,j=1}^M.
\end{equation}
Because of its finite length, $\bsZ(\omega)$ can be expressed exactly (but
not uniquely) as a linear combination of a finite number of i.i.d.\
standard normals, i.e., as
\begin{equation}\label{eq:Gauss_exp1}
 \bsZ(\omega) \,=\,  B \bsY(\omega)\  + \ \bZbar \  ,
 \quad \text{where} \quad  \bsY \sim \cN(\boldsymbol{0}, I_{s \times s}).
\end{equation}
for some real $M\times s$ matrix $B$ with $s\ge M$ satisfying
\begin{equation} \label{eq:Rfactor}
  R \,=\, BB^\top.
\end{equation}
To verify this construction, simply note that \eqref{eq:Gauss_exp1} and
\eqref{eq:Rfactor} imply
$$
  \mathbb{E}[(\bsZ-\bZbar)(\bsZ-\bZbar)^{\top}]
  \,=\, \mathbb{E}[B\bsY \bsY^\top B^{\top}]
  \,=\, B\,\mathbb{E}[\bsY \bsY^\top]B^{\top}=BB^{\top} \,=\, R,
$$
so ensuring that \eqref{eq:covar} is satisfied on the discrete grid.

In principle the factorization \eqref{eq:Rfactor} could be computed via a
Cholesky factorization (or even a spectral decomposition) of $R$ with
$s=M$, but this is likely to be prohibitively expensive, since $R$ is
large and dense. However, under appropriate ordering of the indices, $R$
is a nested block Toeplitz matrix and, as we will explain below, $R$ can
be embedded in a bigger $s \times s$ nested block circulant matrix whose
spectral decomposition can be rapidly computed using FFT with $\calO(s
\log s)$ complexity. A subtle but vital point is that while the covariance
matrix $R$ is automatically positive semidefinite, the extension to a
larger nested block circulant matrix may lose definiteness, yet the nested
block circulant matrix must be at least positive semidefinite for the
algorithm to work. Small deviations from positive
semidefiniteness are acceptable if one is prepared to accept the
incurred errors from omitted negative eigenvalues. That error can be
controlled via an a posteriori bound in terms of the negative
eigenvalues (see, e.g.,  \cite[\S 6.5]{LPS14}). In the present paper,
however, we insist on positive definiteness.

The principle of our extension of the covariance matrix $R$ is as follows.
(Full details are in \S\ref{subsec:extended} below.) We first embed the
unit cube $[0,1]^d$ in a larger cube $[0, \len]^d$ with side length $\len
= m h_0\ge 1$ for some integer $m\ge m_0=1/h_0$. Note that $\rho$ is
automatically defined on $[0,\len]^d$, since it is defined on all of
$\mathbb{R}^d$. Then we construct a $2\len$-periodic even symmetric
extension of $\rho$ on $[0,2\len]^d$, called $\rhoext$, see
\eqref{eq:rhotildef} below, that coincides with $\rho$ on $[0,\len]^d$.
The extended $s\times s$ matrix with
\begin{align} \label{eq:defs}
 s=(2m)^d
\end{align}
is then obtained by the analogue of formula~\eqref{eq:Rmatrix}, with
$\rho$ replaced by $\rhoext$. In Theorem \ref{positivityoftransf}, we
show, under quite general conditions, that if $\len$ (equivalently $m$) is
chosen large enough, this extension is necessarily positive definite. The
algorithm used in practice (Algorithm \ref{alg1} in \S \ref{sec:expect}
below) extends $\len$ cautiously through a sequence of increments in $m$
until positive definiteness is achieved.

To know that the resulting algorithm is efficient, we need a lower bound
on the value of $\len$ needed to achieve positive definiteness. Our second
set of theoretical results provides such bounds for the important \Mat
class of covariance functions, defined in \eqref{defmatern} below. In
Theorem~\ref{cor:matern-growth} we show that positive definiteness is
always achieved  with
\begin{align}
 \ell/\lambda  \ \geq  \  C_1\  + \ C_2\,  \nu^{1/2} \, \log\left( \max\{ {\lambda}/{h_0}, \, \nu^{1/2}\} \right) \ ,
 \label{eq:matlen}
\end{align}
where $\lambda$ is a parameter with the unit of length (the ``correlation
length''),  $\nu<\infty$ is the \Mat smoothness parameter and $C_1, C_2$
are constants independent of $h_0, \len, \lambda, \nu$ and variance
$\sigma^2$. Thus the required $\len$ grows very slowly in $1/h_0$ and can
get smaller as the correlation length decreases. There is some growth as
$\nu$ increases. However, the less smooth fields with $\nu$ small, and
with small correlation length $\lambda$, are the ones often found in
applications -- see the references in \cite{GrKuNuScSl:11}. In Theorem
\ref{thm:PDGauss} we discuss the same question  in the Gaussian case ($\nu
= \infty$). This theorem shows,  for example, that if $\lambda$ and $h_0$
both decrease but $\lambda/h_0$ is kept fixed (i.e., a fixed number of
grid points per unit correlation length), then the minimum value of $\ell$
needed for positive definiteness can decrease linearly in $\lambda$.

An additional benefit of the spectral decomposition obtained by applying
the FFT to the circulant matrix is that it allows us to determine
empirically which variables are the most important in the system. This
gives an ordering of the variables, which can be used to drive the design
of the Quasi-Monte Carlo algorithms (see, e.g., \cite{paper2}).

While circulant embedding techniques are well-known in the computational
statistics literature (e.g., \cite{ChWo:94,ChWo:97,DiNe:97,KrBo:14,
LPS14}), there is relatively little theoretical analysis of this
technique, the best existing references being Chan and Wood \cite{ChWo:94}
and Dietrich and Newsam \cite{DiNe:97}. First, \cite{ChWo:94} provides a
theorem in general dimension $d$ identifying conditions on the covariance
function which ensure that the circulant extension (which we shall
describe below) is positive definite for some sufficiently large~$\ell$.
The condition is similar to that provided in our Theorem
\ref{positivityoftransf} below, except that an additional assumption  on
the discrete Fourier transform (the ``spectral density'') of $\rho$ is
required in \cite{ChWo:94}. In our work we require  positivity of the {\em
continuous} Fourier transform which is automatically satisfied via
Bochner's theorem, due to the fact that $\rho$ is a covariance kernel. On
the other hand Dietrich and Newsam \cite{DiNe:97} provide more detailed
information about the behaviour of the algorithm by restricting the theory
to a 1-dimensional domain. They show that the discrete Fourier transform
of $\rho$ will be positive when $r_{\rm cov}$ is convex, decreasing and
non-negative. This automatically proves the success of the algorithm for
certain covariance kernels. However, two covariance kernels that are not
covered by this theory are the ``Whittle'' covariance and the Gaussian
covariance. These belong to the \Mat class introduced in Example
\ref{ex:Mat} below, with \Mat parameter $\nu = 1$ and $\nu = \infty$
respectively. Our theory covers the whole \Mat family and also describes
the behaviour of the embedding algorithm with respect to the parameters
$\nu$  (\Mat parameter) and  $\lambda$ (correlation length) as well as the
mesh size $h_0$. Finally we note that \cite{DiNe:97} also describes
embedding strategies which are more general than those which we describe
and analyse here, but without theory. % for those.

We may compare the present approach with other methods for
the generation of Gaussian random fields. The principal limitations of
the present approach are that it requires the covariance function to be
stationary, and generates the random field only on a uniform tensor
product grid. The requirement of uniformity of the grid is not very
restrictive for PDE applications, as we show in \cite{paper2} where we
take the interpolation error into account for the error analysis. In
\cite{FeKuSl:17}, the uniform grid requirement is removed but at the
price of approximating the covariance matrix by an $H^2$-matrix, and
then obtaining an approximate square root of that $H^2$-matrix. The use
of $H$-matrices to promote the efficient computation of KL
eigenfunctions has previously been explored in
\cite{EiErUl:07,KoLiMa:09,HPS:15}. In \cite{ScTo:06}, fast multipole
methods were promoted to achieve the same objective. In all these
approaches, there is an inevitable need to truncate the KL series, but
on the other hand there is no essential restriction to stationary
covariance functions. We mention also the recent paper \cite{BaCoMi:16}
in which a related periodization procedure is applied (this time using a
smooth cut-off function), facilitating wavelet expansions of $Z$.

An interesting paper closely related to the present paper
is \cite{HPS:15}, where the authors propose and analyse an approximate,
pivoted Cholesky factorisation followed by a small eigendecomposition,
which provides in many cases a very efficient and accurate way to
approximate the random field. This approach does not require
stationarity or a uniform grid, but for efficiency reasons the
factorisation needs to be truncated after $K \ll M$ steps, where in the
simplest case $M$ is the number of grid points as in the present paper.
For fixed $K$, the accuracy of the truncated expansion depends on the
decay of the KL eigenvalues and the cost is $\mathcal{O}(K^2M)$.  When
the KL eigenvalues do not decay sufficiently fast, this limits the
possible accuracy or leads to prohibitively high costs.

The structure of this paper is as follows. The circulant embedding
algorithm is described in \S\ref{subsec:extended}. The general theory of
positive definiteness of the circulant extension of $R$ is given in
\S\ref{subsec:posdef}. General isotropic covariances (of which the \Mat is
an example) are studied in \S\ref{subsec:isot}, with the estimate
\eqref{eq:matlen} proved in Theorem~\ref{cor:matern-growth}. A key
ingredient in the theory of the related paper \cite{paper2} is an estimate
on the rate of decay of the  eigenvalues of the the circulant extension of
the original covariance matrix. We discuss this question in
\S\ref{sec:decay} where we prove several results about the eigenvalues of
the circulant matrix, leading to the conjecture that they decay at the
same rate as the Karhunen-Lo\`{e}ve eigenvalues of the continuous field.
Numerical experiments illustrating the theory are given in
\S\ref{sec:Numerical}.

% -----------------------------------------------------------------------------
% -----------------------------------------------------------------------------
% -----------------------------------------------------------------------------
% -----------------------------------------------------------------------------
% -----------------------------------------------------------------------------
% -----------------------------------------------------------------------------

\section{Circulant Embedding}
\label{sec:expect}

There are potentially several ways of computing the factorization
\eqref{eq:Rfactor}. Since $R$ is symmetric positive definite, its Cholesky
factorization $R=LL^{\top}$ yields \eqref{eq:Rfactor}, with $s=M$.
Alternatively, we can use the spectral decomposition $R= W \Lambda
W^{\top}$, with $W$ being the orthogonal matrix of eigenvectors, and
$\Lambda$ being the positive diagonal matrix of eigenvalues, again giving
\eqref{eq:Rfactor} with $s=M$, this time with $B = W \Lambda^{1/2}$. Given
that the matrix $R$ is large and dense, in both of these cases the
factorization will generally  be expensive. In the present paper  we use a
spectral decomposition of a circulant extension of $R$ to obtain
\eqref{eq:Rfactor}  with $s>M$ -- see Theorem \ref{thm:decomp} below.

\subsection{The extended matrix}
\label{subsec:extended}

Recall that we began with a uniform grid of points on the $d$-dimensional
unit cube $[0,1]^d$. The $M = (m_0+1)^d$ points $\{\bsx_i: i = 1, \ldots,
M\}$ are assumed to have spacing $h_0 :=1/m_0$ along each coordinate
direction. We then consider an enlarged cube $[0,\len]^d$ of edge length
\begin{align}
\label{eq:defell}
\len:=mh_0\ge 1,
\end{align}
 with integer $m\ge m_0$. We take the point of view
that $m_0$ is fixed (and hence so too is $h_0$), while $m$ is variable
(and hence so too is  $\ell$).

We index the grid points on the unit cube by an integer vector $\bk$,
writing
\begin{equation}\label{eq:gridptsm}
 \bsx_{\bk} \,:=\, h_0 \bk  \qquad\text{for}\quad \bk = (k_1, \ldots, k_d)  \in \{0 , \ldots, m_0\}^d .
\end{equation}
Then it is easy to see that (with analogous vector indexing  for the rows
and columns) the $M \times M$ covariance matrix $R$ defined in
\eqref{eq:Rmatrix} can be written as
\begin{equation} \label{eq:defR}
  R_{\bk, \bk'}
  \,=\, \rho\big({h_0} ({\bk-\bk'}) \big),
  \qquad \bk, \bk'  \in \{ 0, \ldots, m_0\}^d  .
\end{equation}
If the vectors $\bk$ are enumerated in lexicographical ordering, then $R$
is  a nested block Toeplitz matrix where the number of nested levels is
the physical dimension $d$. We remark that all indexing of matrices and
vectors by vector notation in this paper is to be considered in this way.

In the following it will be convenient to extend the definition of the
grid points \eqref{eq:gridptsm} to an infinite grid,
\begin{equation*} %\label{eq:gridpts}
 \bsx_{\bk} \,:=\, h_0 \bk  \qquad\text{for}\quad \bk \in \bbZ^d .
\end{equation*}
Then, in order to define the extended matrix $\Rext$, we define a
$2 \len$-periodic  map on $\bbR$ by specifying its action on $[0,2 \len]$:
\begin{equation*}
  \varphi(x)
  \,:=\,
  \begin{cases}
    x          & \text{if}\quad 0\,\le\, x  \,\le\, \len ,\\
    2\len - x  & \text{if}\quad \len \,\le\, x  \,<\, 2\len .
  \end{cases}
\end{equation*}%
Now we apply this map componentwise and so define an extended version
$\rhoext$ of $\rho$ as follows:
\begin{equation}\label{eq:rhotildef}
 \rhoext(\bsx)
 \,:=\,
 \rho(\varphi(x_1),\ldots,\varphi(x_d))
 ,\qquad \bsx\in \bbR^d .
\end{equation}
Note that $\rhoext$ is $2\len$-periodic in each coordinate direction and
\begin{align} \label{eq:extprop}
 \rhoext(\bsx) \ =\  \rho(\bsx) \quad  \text{when} \quad   \bsx \in [0,\len]^d.
\end{align}
Then $\Rext$ is defined to be the $s\times s$ symmetric nested block
circulant matrix with $s = (2m)^d$, defined, analogously to
\eqref{eq:defR}, by
\begin{equation}\label{eq:Rext_def}
 \Rext_{\bk, \bk'}\,=\, \rho^{\rm{ext}}\big( {h_0} ({\bk - \bk'}) \big),
 \qquad \bk, \bk' \in \{0, \ldots,  2m-1\}^d  .
\end{equation}
Moreover, $R$ is the submatrix of $\Rext$ in which the indices are
constrained to lie in the range $\bk, \bk' \in \{ 0,\ldots, m_0\}^d$.

In the following it will be convenient to introduce the notation, defined
for any integer $m \ge 1$,
\begin{align*} %\label{eq:Z}
  \bbZ_{2m}^d := \{ 0, \ldots, 2m-1 \}^d ,
 \quad  \quad \
  {\overline{\bbZ}_m^d :=
  \{ -m, \ldots ,  m-1\}^d}
  .
\end{align*}
Then we have the following simple result.

\begin{proposition} \label{prop:eigs} $\Rext$ has real eigenvalues
$\Lambdaext_\bk$ and corresponding normalised eigenvectors $\bsV_\bk$,
given, for $\bk \in {\bbZ}_{2m}^d$ by the formulae:
\begin{equation}\label{eq:eivsandeifs}
\Lambdaext_\bk={\sum_{\bk'\in \overline{\bbZ}_{m}^d}}\rho({h_0\bk'})
\exp\bigg({-2\pi \ri\frac{\bk\cdot \bk'}{2m}}\bigg),\quad
\text{and} \quad (\bsV_\bk)_{\bkappa}=\frac{1}{\sqrt{s}}\exp\bigg(2\pi \ri\frac{\bk\cdot\bkappa}{2m}\bigg),
\quad\bkappa \in {\bbZ}_{2m}^d,
\end{equation}
where $s$ is given in \eqref{eq:defs}.
\end{proposition}

\begin{proof}
By \eqref{eq:rhotildef} and the fact that  $\varphi$ is symmetric, we see
that $\rhoext$ is also symmetric, so the eigenvalues of $\Rext$ are real.
To obtain the required formula for the eigenvalues, we use
\eqref{eq:Rext_def} and the formula for the eigenvectors in
\eqref{eq:eivsandeifs} to write, for $\bk, \bkappa \in \bbZ_{2m}^d$.
\begin{align*}
\left(\Rext \bsV_\bk\right)_{\bkappa} \ & =  \ \sum_{\bk' \in \bbZ_{2m}^d}
\rhoext(h_0(\bkappa- \bk')) \left(\bsV_\bk\right)_{\bk'}
\ = \ \sum_{\bk'' \in \bkappa - \bbZ_{2m}^d }
\rhoext(h_0 \, \bk'') \left(\bsV_\bk\right)_{\bkappa - \bk''}  \\
\ & = \ \Bigg( \sum_{\bk'' \in \bkappa - \bbZ_{2m}^d }
\rhoext(h_0 \, \bk'')  \exp\left(- 2 \pi \ri \frac{\bk\cdot \bk''}{2m} \right) \Bigg)
\left(\bsV_\bk\right)_{\bkappa} .
\end{align*}
Then \eqref{eq:eivsandeifs} follows, on using the coordinatewise
$2m$-periodicity of the summand in the last equation, and also the
extension property \eqref{eq:extprop}.
\end{proof}

It follows from the simple form of the eigenfunctions that the matrix
$\Rext$ can be diagonalised by FFT. The following version of the spectral
decomposition theorem, taken from \cite{GrKuNuScSl:11}, has the advantage
that it allows the diagonalisation to be implemented using only real FFT.
\begin{theorem} \label{thm:decomp}
$\Rext$ has the spectral decomposition:
\begin{align*} %\label{eq:Rext}
\ \Rext \ =\  \Qext \Lambdaext \Qext ,
\end{align*}
where
$\Lambdaext$ is the diagonal matrix containing the eigenvalues of $\Rext$,
  and $\Qext = \Re (\cF) + \Im (\cF)$ is real symmetric,  with
$$
\cF_{\bsk, \bsk'} =  \frac{1}{\sqrt{s}} \exp \left({2 \pi} \ri \frac{\bsk'
\cdot \bsk}{{2 m}} \right)   $$ denoting the {$d$-dimensional}  Fourier
matrix. If the eigenvalues of $\Rext$ are all non-negative then the
required $B$ in \eqref{eq:Rfactor} can be obtained by selecting $M$
appropriate rows of
\begin{equation} \label{eq:defBext}
 \Bext: =  \Qext (\Lambdaext)^{1/2} .
\end{equation}
\end{theorem}

The use of FFT allows fast computation of the matrix-vector product
$\Bext\bsy$ for any vector $\bsy$, which then yields $B\bsy$ needed for
sampling the random field in \eqref{eq:Gauss_exp1}. Our algorithm for
obtaining a minimal positive definite $\Rext$ is given in
Algorithm~\ref{alg1}. Our algorithm for sampling an instance of the random
field is given in Algorithm~\ref{alg2}.
Note that the normalisation used within the FFT routine differs among
particular implementations. Here, we assume the Fourier transform to be
unitary.

We can replace Step~1 of Algorithm~\ref{alg2} with a QMC point from
$[0,1]^s$ and mapped to $\bbR^s$ elementwise by the inverse of the
cumulative normal distribution function. The relative size of the
eigenvalues in $\Lambdaext$ tells us the relative importance of the
corresponding variables in the extended system, which helps to determine
the ordering of the QMC variables.

\begin{algorithm}
\algcaption{\label{alg1}}
Input: $d$, $m_0$, and covariance function $\rho$.
\begin{enumerate}
\item Set $m = m_0$.
\item Calculate $\bsr$, the first column of $\Rext$ in
    \eqref{eq:Rext_def}.
\item Calculate $\bsv$, the vector of eigenvalues of $\Rext$, by
    $d$-dimensional FFT on $\bsr$.
\item If smallest eigenvalue $<0$ then increment $m$ and go to Step~2.
\end{enumerate}
Output: $m$, $\bsv$.
\end{algorithm}

\begin{algorithm}
\algcaption{\label{alg2}}
Input: $d$, $m_0$, mean field $\overline{Z}$, and
$m$ and $\bsv$ obtained by Algorithm~\ref{alg1}.
\begin{enumerate}
\item With $s = (2m)^d$, sample an $s$-dimensional normal random vector $\bsy$.
\item Update $\bsy$ by elementwise multiplication with $\sqrt{\bsv}$.
\item Set $\bsw$ to be the $d$-dimensional FFT of $\bsy$.
\item Update $\bsw$ by adding its real and imaginary parts.
\item Obtain $\bsz$ by extracting the appropriate $M=(m_0+1)^d$
    entries of $\bsw$.
\item Update $\bsz$ by adding $\overline{Z}$.
\end{enumerate}
Output: $\bsz$ (or $\exp(\bsz)$ in the case of lognormal field).
\end{algorithm}

In the following subsection we shall show (under mild conditions) that
Algorithm~\ref{alg1} will always terminate. Moreover we shall give (for
the case of the \Mat covariance function) a detailed analysis of how
$\len$ depends on various parameters of the field. Then in
\S\ref{sec:decay}, we give an analysis of the decay rates of the
eigenvalues of $\Rext$ compared with that of the the eigenvalues of the
original Toeplitz matrix $R$ and the KL eigenvalues of the underlying
continuous field $Z$.

\subsection{Positive definiteness}
\label{subsec:posdef}

We first note that by definition \eqref{eq:covar} and \eqref{eq:defrho1}
it follows that for all $N\ge 1$, all point sets $\bst_1,
\ldots,\bst_N\in\mathbb{R}^d$, and all $\bsgamma\in\mathbb{R}^N$
\begin{equation}\label{eq:quadfm}
  \sum_{i=1}^N\sum_{i'=1}^N \gamma_i\,\gamma_{i'}\,\rho(\bst_i-\bst_{i'}) \ =\  \bbE[ (W - \overline{W})^2] \ \geq\  0,
\end{equation}
where $$W(\omega) = \sum_{i=1}^N \gamma_i Z(\bst_i, \omega)$$ and
$\overline{W} $ denotes its mean. This observation, together with
\eqref{eq:symm}, means that $\rho$ is a {\em symmetric} {\em positive
semidefinite} function (see, e.g., \cite[Chapter~6]{wendland05}). If in
addition, \eqref{eq:quadfm} is always positive for nonzero $\bsgamma$ then
$\rho$ is called  {\em positive definite}. In our later applications,
$\rho$ will always be positive definite.

We use the following definition of the Fourier transform of an
absolutely integrable function~$\rho$:
\begin{equation}\label{eq:FT}
\wrho(\bsxi) \ = \
\int_{\bR^d} \rho(\bsx) \,\exp(-2\pi\ri\,\bsxi\cdot\bsx)  \, \rd \bsx \ , \qquad
\bsxi \in \bR^d\ .
\end{equation}
If, in addition, $\rho$ is continuous and $\wrho$ is absolutely integrable
then the Fourier integral theorem gives
\begin{equation*}%\label{eq:IFT}
\rho(\bsx) \ = \
\int_{\bR^d} \wrho(\bsxi) \,\exp(2\pi\ri\,\bsxi\cdot\bsx) \,\rd \bsxi \ , \qquad
\bsx \in \bR^d\ .
\end{equation*}
In this case, $\wrho$ is also continuous, and Bochner's theorem (e.g.,
\cite[Theorem 6.3]{LPS14}) states that $\rho$ is a positive definite
function if and only if $\wrho$ is positive.

The following theorem is the main result of this subsection.

\begin{theorem}\label{positivityoftransf}
Suppose that $\rho \in L^1(\mathbb{R}^d)$ is a real-valued, symmetric
positive definite function with the additional reflectional symmetry
\begin{equation*} %\label{eq:reflect}
  \rho(\bsx) = \rho(\pm x_1,\ldots,\pm x_d) = \rho(|x_1|,\ldots,|x_d|)
  \qquad\forall \bsx\in\bbR^d,
\end{equation*}
and suppose $\wrho \in L^1(\mathbb{R}^d)$. Suppose also that for the given
value of $h_0=1/m_0$ we have
\begin{equation}\label{eq:conditiononrho}
 \sum_{\bk \in \bbZ^d} \vert \rho(h_0 \bk )\vert \ <\ \infty.
\end{equation}
Then Algorithm~\ref{alg1} will always terminate with a finite value of
$m$, and the resulting matrix $R^{\rm{ext}}$ will be positive definite.
\end{theorem}

We prove Theorem~\ref{positivityoftransf} via a technical estimate --
Lemma~\ref{lem:tech} -- which provides an explicit lower bound for the
eigenvalues of $\Rext$. This estimate proves the theorem, and moreover,
allows us to obtain explicit lower bounds for the eigenvalues  in the
lemmas which follow.

\begin{lemma} \label{lem:tech}
Under the assumptions of Theorem~\ref{positivityoftransf}, the eigenvalues
$\Lambdaext_\bk$ of $\Rext$ all satisfy the estimate
\begin{equation} \label{eq:eigest}
 \Lambdaext_\bk \ \geq\ \frac{1}{h_0^d}
  \min_{\bszeta\in[-\frac{1}{2},\frac{1}{2}]^d}
  \sum_{\br\in\bbZ^d}\wrho\Big(\frac{\bszeta+\br}{h_0}\Big)
   - \sum_{\bk' \in \bbZ^d\setminus\overline{\bbZ}_{m}^d}
   | \rho(h_0 \bk')|\ .
\end{equation}
\end{lemma}

\begin{proof}
By Proposition ~\ref{prop:eigs}, we have
\begin{align}
 {\Lambda_{\bk}^{\rm{ext}}} &\,=\,
 \sum_{\bk' \in \bbZ^d}
 \rho\big(h_0 \bk' \big)
 \exp\left( - {2\pi\ri} \frac{\bk' \cdot \bk }{2m} \right)
 - \sum_{\bk' \in \bbZ^d\setminus\overline{\bbZ}_{m}^d}
 \rho\big(h_0 \bk' \big)
 \exp\left(  - {2 \pi\ri} \frac{\bk' \cdot \bk}{2m}  \right)
 . \label{eq:condition}
\end{align}
We obtain the lower bound on the first sum on the right-hand side of
\eqref{eq:condition} using \eqref{eq:min-bound} of
Theorem~\ref{thm:sampling} in the Appendix with $h= h_0$ and
$\boldsymbol{\xi} =  \bk'/(2m)$, and an upper bound on the second sum in
the obvious way.
\end{proof}

\begin{proofof}{{Theorem~\ref{positivityoftransf}}}
Because of the assumed positive definiteness of $\rho$, it follows from
\linebreak Bochner's theorem that the Fourier transform $\wrho$ is
positive, so the strict positivity of the first term on the right-hand
side of \eqref{eq:eigest} follows from Theorem \ref{thm:sampling}. The
lower bound is also independent of $m$. For fixed $h_0$,
\eqref{eq:conditiononrho} ensures that the tail sum in the second term on
the right-hand side of \eqref{eq:eigest} converges to zero as
$m\rightarrow \infty$. Hence the result follows.
\end{proofof}

\subsection{Isotropic covariance}
\label{subsec:isot}

More detailed lower bounds on $\Lambdaext_{\bk}$  can be obtained under
the additional assumption that
 the random field $Z$ in \eqref{eq:covar} is {\em
isotropic}, i.e.,
\begin{equation} \label{eq:kwalpha}
  \rho(\bsx) \,=\, \kappa(\Vert \bsx \Vert_2/\lambda)\ ,
\end{equation}
where the parameter $\lambda$ is a correlation length which will play a
key role in Example \ref{ex:Mat} and Theorem \ref{cor:matern-growth}
below. In this case the Fourier transform \eqref{eq:FT} is given by
\begin{align} \label{eq:scaledFT}
 \widehat{\rho}(\bsxi)  \,=\,
\lambda^d\, \widehat{\kappa}_d({\lambda \Vert \bsxi\Vert_2 })\ ,
\end{align}
where
\begin{align}\label{eq:Hankel}
  \widehat{\kappa}_d(r)
  \,:=\,
  \frac{2 \pi}{r^{(d-2)/2}} \int_0^\infty \kappa(t) \, t^{d/2} \, J_{(d-2)/2}(2\pi \, r t) \, \rd{t}
  \qquad
  \text{with } r  \ge 0
  \ ,
\end{align}
and $J_{\alpha}$ denotes the Bessel function of order $\alpha$. The
right-hand side of \eqref{eq:Hankel} is the \emph{Hankel transform} of
$\kappa$ (see, e.g., \cite[Theorem 1.107]{LPS14}). The following lemma
then estimates each of the terms on the right-hand side of
\eqref{eq:eigest} to obtain an explicit lower bound on the eigenvalues of
$\Rext$.

\begin{lemma} \label{lem:bits}
Suppose $\rho$ and  $\wrho$ are given as in
\eqref{eq:kwalpha}--\eqref{eq:Hankel}.
\begin{enumerate}
\item [\textnormal{(i)}] %
If $\,\vert \kappa \vert$ is a decreasing function on $\bR_+$ and
$\,r^{d-1} \,\kappa(r) \in L_1(\bR_+)$, then  (with $\len = m h_0$),
\[ \sum_{\bk \in {\bbZ^d\setminus\overline{\bbZ}_{_m}^d}} \vert
  \rho(h_0 \bk) \vert \,\le\, \frac{(3^d-1)2^{d-1}}{(h_0/\lambda)^d}
  \int_{(\len- h_0)/\lambda}^\infty r^{d-1} \, |\kappa(r)| \,\rd r\, .
\]
\item [\textnormal{(ii)}] %
If $\widehat{\kappa}_d$ is positive and decreasing on $\bR_+$ and
$r^{d-1}\,\widehat{\kappa}_d(r) \in L_1(\bR_+)$,
then
\[
  \frac{1}{h_0^d} \min_{\bszeta\in [-\frac{1}{2},\frac{1}{2}]^d}
  \sum_{\bsr \in \bbZ^d} \wrho\Big(\frac{\bszeta + \br}{h_0}\Big)
  \,\ge\, \frac{2^d}{d^{d/2-1} \, 3^{d-1}} \int_{3 d^{1/2} \lambda/
  (2 h_0)}^\infty r^{d-1}\, \widehat{\kappa}_d(r) \,\rd r  \, .
\]
\end{enumerate}
\end{lemma}

\begin{proof}
We begin by deriving  upper and lower bounds for some elementary
sequences.

Suppose $g$ is any positive and decreasing function satisfying $r^{d-1}
g(r)\in L_1(\bbR_+)$. {For} any integer $m\ge 1$ we have
\begin{align}\label{eq:sum:ub}
  \sum_{j=m+1}^\infty j^{d-1} \, g(j)
  \le
  \sum_{j=m+1}^\infty \int_{j-1}^j \ceil{r}^{d-1} \, g(r) \, \rd{r}
  \le
  \int_m^\infty (2r)^{d-1} \, g(r) \, \rd{r}
  ,
\end{align}
where the first inequality uses $\ceil{r} = j$ for $r \in (j-1,j]$ and the
second inequality follows from $\ceil{r} \le 2r$ for $r\ge1$. Similarly,
we can obtain a lower bound for $m \geq 1$,
\begin{align}\label{eq:sum:lb}
  \sum_{j=m+1}^\infty j^{d-1} \, g(j)
  \ge
  \sum_{j=m+1}^\infty \int_j^{j+1} \floor{r}^{d-1} \, g(r) \, \rd{r}
  \ge
  \int_{m+1}^\infty (r/2)^{d-1} \, g(r) \, \rd{r}
  .
\end{align}

Also let us consider the set $S_d(j) := \{ \bk \in \bbZ^d: \|\bk \|_\infty
= j \}$ for any integer $j\ge 1$. Then, with $\# \, S_d(j)$ denoting the
cardinality of this set, we have the bounds
\begin{align}\label{eq:Sdj:lb:ub}
  d \, 2^d \, j^{d-1}
  \ \le
\  \# \, S_d(j)
  \ = \ (2j+1)^d-(2j-1)^d
  \ = \ 2 \sum_{\substack{i=1 \\ i \text{ odd}}}^d \binom{d}{i} (2j)^{d-i}
  \ \le \
  (3^d-1) j^{d-1},
\end{align}
where the lower bound is the $i=1$ term in the binomial expansion and the
upper bound comes from the  estimate $j^{d-i} \le j^{d-1}$ for $i\ge 1$.
The bounds are  exact for $d = 1$ and $d=2$, with $\# \, S_1(j) = 2$ and
$\# \, S_2(j) = 8j$, while for $d=3$ we have  $ \# S_3(j) = 2 + 24j^2$,
and the lower and upper bounds given by \eqref{eq:Sdj:lb:ub}  are $24 j^2$
and $ 26 j^2$ respectively.

Now, to prove (i), using \eqref{eq:Sdj:lb:ub} and \eqref{eq:sum:ub} and
$\|\bk\|_2 \ge \|\bk\|_\infty$, we can now write
\begin{align*}
 \sum_{\bk \in {\bbZ^d\setminus \overline{\bbZ}_{m}^d}} & \vert \kappa(h_0
\|\bk\|_2/\lambda) \vert
 \,\leq \, \sum_{j = m}^\infty  \sum_{\Vert \bk \Vert_\infty = j  } \vert
 \kappa( h_0  \|\bk\|_2/\lambda) \vert
 \,\le\, (3^d-1) \sum_{j=m}^\infty \, j^{d-1} \,
 \vert \kappa(h_0 j/\lambda )\vert    \\
 &\,\le\, (3^d-1)2^{d-1} \int_{m-1}^{\infty} r^{d-1} \, \vert \kappa(h_0 r /\lambda) \vert \,\rd r
 \,=\, \frac{(3^d-1)2^{d-1}}{(h_0/\lambda)^d} \int_{(\len-h_0)/\lambda }^{\infty} r^{d-1} \,\vert \kappa(r)\vert  \,\rd r,
\end{align*}
with $\len=mh_0$, thus completing the proof of (i).

To prove (ii), note first that, for $\bszeta \in [-\frac{1}{2},\frac{1}{2}]^d$ and any
$\br \in \bbZ^d$ with $\Vert \br \Vert_\infty = j$, we have $\|\bszeta +
\br\|_2 \le d^{1/2} \|\bszeta + \br\|_\infty \le d^{1/2} (1/2 + j)$. Using
\eqref{eq:scaledFT}, \eqref{eq:Sdj:lb:ub} and \eqref{eq:sum:lb} and dropping the $\br=\bszero$
term, we can write
\begin{multline*}
  \frac{1}{h_0^d}\sum_{\bsr\in\bbZ^d}
  \wrho\Big(\frac{\bszeta+\bsr}{h_0}\Big)
  \ =\
  \bigg(\frac{\lambda}{h_0}\bigg)^d\sum_{\br \in \bbZ^d}
    \widehat{\kappa}_d\Big(\frac{\lambda\|\bszeta + \br\|_2}{h_0}\Big)
  \ \ge \
  \bigg(\frac{\lambda}{h_0}\bigg)^d\sum_{j=1}^\infty \sum_{\|\br\|_\infty=j}
    \widehat{\kappa}_d\Big(\frac{\lambda\|\bszeta + \br\|_2}{h_0}\Big)
  \\
\   \ge\
  \frac{d\,2^d}{(h_0/\lambda)^d} \sum_{j=1}^\infty j^{d-1}
    \widehat{\kappa}_d\Big(\lambda\frac{d^{1/2} (1/2+j)}{h_0}\Big)
  \ \ge
\   \frac{2d}{(h_0/\lambda)^d} \int_1^\infty r^{d-1}
    \widehat{\kappa}_d\Big(\frac{\lambda d^{1/2} (1/2+r)}{h_0}\Big) \, \rd r
  \\
\   = \
  \frac{2}{d^{d/2-1}} \int_{{3 d^{1/2} \lambda/(2 h_0)}}^\infty
  \left(r - \frac{\lambda d^{1/2}}{2h_0}\right)^{d-1}
    \widehat{\kappa}_d(r) \, \rd r
 \  \ge
\  \frac{2}{d^{d/2-1}} \int_{{3 d^{1/2} \lambda/(2 h_0)}}^\infty \left(\frac{2r}{3}\right)^{d-1}
    \widehat{\kappa}_d(r) \, \rd r ,
\end{multline*}
where the last inequality follows from $r - c \ge 2r/3 \Leftrightarrow r
\ge 3c$, with $c= d^{1/2}\lambda / (2h_0)$.
\end{proof}

\begin{corollary}
\label{cor:PD_isot} Under the assumptions of Lemma~\ref{lem:bits}, $\Rext$
is positive definite if
\begin{equation}
  \int_{3 \lambda d^{1/2} / (2h_0)}^\infty r^{d-1}\, \widehat{\kappa}_d(r) \,\rd r
  \, > \,
  \frac{(3^d-1)\,3^{d-1}\, d^{d/2-1}}{2 (h_0/\lambda)^d } \int_{(\len-h_0)/\lambda}^\infty r^{d-1} \, |\kappa(r)| \,\rd r
  .
 \label{eq:corest}
\end{equation}
\end{corollary}

\begin{proof}
We make use of Lemma \ref{lem:tech}. The fact that $\rho\in L_1(\bbR^d)$
and $\wrho\in L_1(\bbR^d)$ follow immediately from \eqref{eq:kwalpha} and
\eqref{eq:Hankel} respectively, and from the assumptions
 $r^{d-1}\kappa(r)\in L_1(\bbR^+)$ and
$r^{d-1}\widehat{\kappa}_d(r)\in L_1(\bbR^+)$. It then follows from Part
(i) of Lemma~\ref{lem:bits} that the assumption \eqref{eq:conditiononrho}
of Theorem~\ref{positivityoftransf} is satisfied. Since the symmetry
assumption in the theorem is automatically satisfied by an isotropic
covariance, the result now follows immediately from Lemma~\ref{lem:tech}.
\end{proof}

To interpret Corollary~\ref{cor:PD_isot}, recall that $\lambda$ is the
correlation length, so $\lambda$ is bounded above (without loss of
generality let us assume $\lambda \le 1$), but $\lambda$ may approach $0$.
If $h_0$ is chosen so that $h_0/\lambda$ is a fixed constant, then
\eqref{eq:corest} and integrability of $r^{d-1} \vert \kappa(r)\vert$
ensures that positive definiteness is achieved for $\len$ sufficiently
large independently  of $\lambda$ and $h_0$. The condition ``$h_0/\lambda$
constant'' is a natural analogue to the requirement in oscillatory
problems that the meshwidth should be proportional to the wavelength.
However if $\len$ and $\lambda$ are  fixed and  $h_0 \rightarrow 0$
(constituting ``subwavelength mesh refinement'' needed to get higher
accuracy) we see that the sufficient condition that ensures positive
definiteness will eventually fail. One interesting question is how fast
$\len$ needs to grow as $h_0$ decreases in order to be sure of positive
definiteness. This can be answered fairly completely and satisfactorily in
the case of the \Mat family.

\begin{example} \label{ex:Mat}
The \Mat family of covariances are defined by
\begin{equation} \label{defmatern}
 \rho(\bsx) =
 \kappa(\Vert \bsx \Vert_2/\lambda) \ ,  \quad \text{where} \quad  \kappa(r)  = \
\sigma^2 \, \frac{2^{1-\nu}}{\Gamma(\nu)}
(\sqrt{2\nu}\,  r)^{\nu} K_\nu\left( \sqrt{2\nu}\, r \right).
\end{equation}
Here $\Gamma$ is the gamma function and $K_\nu$ is the modified Bessel
function of the second kind, $\sigma^2$ is the variance, $\lambda$ is the
correlation length and $\nu > 0$ is a smoothness
parameter. The practically interesting range is $\nu \ge
1/2$, with the cases $\nu = 1/2$ and $\nu = \infty$ corresponding to
the exponential and Gaussian covariances respectively, see, e.g.,
\cite{GKNSSS:15}. For this reason, we have restricted the
analysis to the case $\nu \geq 1/2$, see also Remark
\ref{rem:Matern_extend}. 

For this particular $\kappa$, the Hankel transform
$\widehat{\kappa}_d$ (see \eqref{eq:Hankel}) is explicitly known, and so
we have from \eqref{eq:scaledFT}--\eqref{eq:Hankel} that
\begin{align}
 \wrho(\bsxi) = \lambda^d \, \widehat{\kappa}_d(\lambda \Vert \bsxi \Vert_2) \ ,
 \quad \text{where} \quad
  \widehat{\kappa}_d(r)
   =
  \sigma^2 2^d \pi^{d/2} (2 \nu)^{\nu} \frac{\Gamma(\nu+ d/2)}{\Gamma(\nu)}
   \, \frac{1}{(2 \nu + (2 \pi r)^2)^{\nu+d/2}} \ .
 \label{eq:Hankel1}
\end{align}
(This can be obtained, for example, by some manipulation of the formula
\cite[p.264]{LPS14}.) Note that for the \Mat case both $\rho$ and $\wrho$
are positive, radial, decreasing functions.
\end{example}

The following theorem shows that when $\nu$ and $\lambda$ are fixed,
$\len$ needs to grow with order $\log h_0^{-1}$ in order to be sure of
positive definiteness as $h_0 \rightarrow 0$. On the other hand if
$\lambda$ and $h_0$ are fixed, $\ell$ needs to grow like $\nu^{1/2}\log
\nu $ as $\nu$ increases. For given $\ell$, $h_0$, the bound on $\ell$
gets smaller as $\lambda$ decreases. This provides conditions guaranteeing
the termination of Algorithm~\ref{alg1}.

\begin{notation}\label{not:tilde}
When discussing the \Mat case, we shall use the notation $A\lesssim B$
(equivalently $B \gtrsim A$)  to mean $A/B$ is bounded above independently
of $\ell, h_0, \lambda$ and $\nu$, and we write $A\sim B$ if $A \lesssim
B$ and $B \lesssim A$.
\end{notation}

\begin{theorem}\label{cor:matern-growth}
Consider the \Mat covariance family \eqref{defmatern}, with smoothness
parameter $\nu$ satisfying $1/2 \leq \nu < \infty$ and correlation length
$\lambda \leq 1$. Suppose $h_0/\lambda \leq e^{-1}$. Then there exist
positive constants  $C_1$ and $C_2\geq 2 \sqrt{2}$ which may depend on
dimension $d$ but are independent of the other parameters $\ell, h_0,
\lambda, \nu, \sigma^2$, such that $\Rext$ is positive definite if
\begin{align}
\ell/\lambda  \ \geq  \  C_1\  + \ C_2\,  \nu^{1/2} \, \log\left( \max\{ {\lambda}/{h_0}, \, \nu^{1/2}\} \right) \ .
\label{eq:alphah0small}
\end{align}
\end{theorem}

\begin{proof}
For convenience we introduce the notation
$$\Psi = \max\{ {\lambda}/{h_0}, \, \nu^{1/2}\}\ .$$

Aiming to verify \eqref{eq:corest}, we note first that, since both
$\kappa$ and $\hat{\kappa}_d$ depend linearly on $\sigma^2$, we can
without loss of generality set  $\sigma^2 = 1$. We shall obtain  the
following lower bound on the left-hand side of \eqref{eq:corest}:
\begin{align}
  \int_{3\lambda d^{1/2}/(2 {h_0})}^\infty r^{d-1}\, \widehat{\kappa}_d(r) \,\rd r   & \ \gtrsim\
\ \nu^{\nu + d/2 -1} ( 4 \pi^2 )^{-\nu} \Psi^{-2 \nu} \ , \label{eq:lhbelow}
\end{align}
and (subject to assumption \eqref{eq:alphah0small}), the following upper
bound on the integral on the right-hand side of \eqref{eq:corest}:
\begin{align}\label{eq:nowRHS2}
 \int_{(\ell-h_0)/\lambda}^\infty r^{d-1} \vert \kappa (r) \vert \rd r
\ \lesssim \  5^{\nu}
\nu^{d/2-1} \exp\left(- \sqrt{\frac{\nu}{2}} \frac{\ell - h_0}{\lambda}\right) \ .
\end{align}

Since these estimates are rather technical, we defer their justification
until the end of this proof. Thus, assuming \eqref{eq:lhbelow} and
\eqref{eq:nowRHS2} we see that there exists $D > 0$, with $D$ independent
of $\ell, h_0, \lambda, \nu$, such that \eqref{eq:corest} holds if
$$
\exp\left(- \sqrt{\frac{\nu}{2}} \frac{\ell-h_0}{\lambda} \right)
\ \le D \  \left( \frac{\nu}{20  \pi^2} \right)^\nu  \Psi^{-2 \nu} \ .
$$

Taking logs and rearranging, this is equivalent to
$$
 (\len- h_0)/\lambda \ \geq \sqrt{2/\nu}\log{(1/ D)}\ + \sqrt{2} \nu^{1/2} \log(20 \pi^2/\nu) +2\sqrt{2} \nu^{1/2} \log \Psi \ .
$$
The first term on the right-hand side is positive only when $D <1$, in
which case it is sufficient to replace $\sqrt{2/\nu}$ by its maximum value
$2$ (given that $\nu\ge 1/2$). Similarly, the second term is positive only
when $\nu \leq 20 \pi^2$, and in this case we have
\begin{align*}
 \nu^{1/2} \log(20 \pi^2/\nu) \ \leq \ \sqrt{20\pi^2} \, \log(40\pi^2) .  %\label{eq:C1}
\end{align*}
Taking into account that $h_0/\lambda \leq e^{-1}$, we have thus
demonstrated the sufficiency of a condition of the form
\eqref{eq:alphah0small}. We now complete the proof by proving  the
technical estimates \eqref{eq:lhbelow} and \eqref{eq:nowRHS2}.
\vspace{0.5cm}

\noindent {\em  Proof of estimate  \eqref{eq:lhbelow}}: \ \ Recalling that
$\widehat{\kappa}_d$ is given by \eqref{eq:Hankel1}, we may use the
following elementary result to bound the ratio of gamma functions if $d$
is even:
\begin{align} \label{eq:Gamma1}
\frac{\Gamma(\nu+ d/2)}{\Gamma(\nu)} \ =  (\nu+d/2-1)\ldots (\nu+1)\nu
\ \geq \ \nu^{d/2}.
\end{align}
If $d$ is odd then we may use
\begin{align}
\frac{\Gamma(\nu+ d/2)}{\Gamma(\nu)} \ & =  (\nu+d/2-1)\ldots (\nu+3/2)(\nu+1/2)
\frac{\Gamma(\nu+1/2)}{\Gamma(\nu)} \label{eq:Gamma2}\\
 &\ge \
\nu^{(d-1)/2}\frac{\Gamma(\nu+1/2)}{\Gamma(\nu)}\ge \
\nu^{(d-1)/2}(\nu-1/2+1/4)^{1/2}\ \ge \ \nu^{(d-1)/2} \nu^{1/2}/\sqrt{2} \
\nonumber
\end{align}
for all $\nu\ge 1/2$, where in the penultimate step we use Kershaw's
inequality (see \cite{Ker}, equation (1.3))
$$
\frac{\Gamma(x+1)}{\Gamma(x+r)}>\left(x+\frac{r}{2}\right)^{1-r} \quad
\mbox{for } \ x>0, \quad  0<r<1,
$$
with $r=1/2$ and $x=\nu-1/2$.  (If $\nu = 1/2$ the result is obtained by
taking the limit $\nu \to 1/2+$ and using continuity.)

On using this lower bound in \eqref{eq:Hankel1}, we obtain
\begin{align*} %\label{eq:hkbelow}
 \widehat{\kappa}_d(r) \ & \gtrsim\ (2 \nu)^\nu \nu^{d/2} (2 \nu)^{-(\nu + d/2)}
 (1 + 2 \pi^2  r^2 / \nu)^{-(\nu + d/2)} \
 \sim  \  (1 + 2 \pi^2  r^2 / \nu)^{-(\nu + d/2)} .
\end{align*}

Now (for convenience), we introduce the notation
$$ a = 3\sqrt{d}/2,
\quad \wM = \max\{ \lambda a / h_0, \sqrt{\nu/2\pi^2}\}.$$
Then the left-hand side of \eqref{eq:corest} can be estimated from below
by:
\begin{align*}
  \int_{\lambda a/h_0}^\infty r^{d-1}\, \widehat{\kappa}_d(r) \, \rd r  \
  \gtrsim  \
 \int_{\lambda a/h_0}^\infty r^{d-1} \left(1 + 2 \pi^2 {r^2}/\nu   \right)^{-(\nu + d/2)} \rd r
    \nonumber \\
 \geq\
 \int_{\wM}^\infty r^{d-1} \left(1 + 2 \pi^2 r^2/\nu   \right)^{-(\nu + d/2)} \rd r \ .
 %\label{eq:int1}
\end{align*}
Now, noting that when $r \geq \wM$ we have $2 \pi^2 r^2/\nu \geq 1$, and
so
\begin{align*}
  \int_{\lambda a/h_0}^\infty r^{d-1}\, \widehat{\kappa}_d(r) \,\rd r   & \geq  \
\int_{\wM}^\infty r^{d-1} (4 \pi^2 r^2 / \nu)^{-(\nu + d/2)} \rd r
\nonumber \\ & =
\left( \frac{\nu}{4 \pi^2}\right)^{\nu + d/2} \int_{\wM} ^\infty r^{-2 \nu -1} \rd r \
\sim  \ \nu^{\nu + d/2 -1} ( 4 \pi^2 )^{-(\nu + d/2)} \wM^{-2 \nu} \ . %\label{eq:int2}
\end{align*}
which yields \eqref{eq:lhbelow}, since $\wM \sim \Psi$.

\vspace{0.5cm}

\noindent {\em Proof of estimate \eqref{eq:nowRHS2}}: \ \ It is sufficient
to prove this estimate wth $\ell- h_0$ on each side replaced by $\ell$.
By definition \eqref{defmatern} and the change of variable $r \mapsto
(\sqrt{2/\nu}) r$, we have
\begin{align}
\int_{\ell/\lambda}^\infty r^{d-1} \vert \kappa (r) \vert \rd r
\ \sim \ \frac{1}{{\Gamma(\nu)}} \left(\frac{\nu}{2}\right)^{\nu + d/2} \int_{\sqrt{2/\nu}(\ell/\lambda)}^\infty
r^{d + \nu - 1 } \vert K_{\nu} (\nu r)\vert  \rd r \ .
\label{eq:nowRHS} \end{align}

To obtain an estimate for the Bessel function on the right-hand side of
\eqref{eq:nowRHS}, note first that by the hypotheses of this theorem, we
have $h_0/\lambda \leq e^{-1}$ and $\nu \geq 1/2$. Now assume that
\eqref{eq:alphah0small} holds, with $C_1>0$ and $C_2>2\sqrt{2}$, both not
yet fixed. Then
\begin{align} \label{eq:alest}
 \ell/\lambda \  \geq \ C_1 + C_2 \nu^{1/2}  \ \geq \ C_1 + 2 \sqrt{2} \nu^{1/2} \
 > \ 2 \sqrt{2} \nu^{1/2} = \, 4\sqrt{\nu/2},
\end{align}
and hence the range of integration in \eqref{eq:nowRHS} is contained in
$[4, \infty)$.

The uniform asymptotic estimate for $K_\nu(\nu r)$ given in
\cite[9.7.8]{AS} implies  that there exists a $\nu^* < \infty $ and a
constant $C$ such that, for all $\nu \geq \nu^*$,
\begin{align}
\label{eq:estBessel}  \vert K_\nu(\nu r) \vert  \
\leq \ C \nu^{-1/2} \exp(-\nu r)  \quad r \in [4, \infty).
\end{align}
We shall show that in fact such an inequality holds for all $\nu\ge 1/2$
and all $r\ge 4$. To this end, for $\nu\ge 1/2$, we introduce the quantity
$$
c(\nu) := \Vert \exp(z) z^{1/2} K_\nu(z) \Vert_{L_{\infty}(2, \infty)},
$$
The continuity of $K_\nu$ on $[2, \infty)$ and the order-dependent
asymptotics of $K_\nu$ \cite[9.7.2]{AS} ensure that $c(\nu) < \infty$ for
all $\nu<\infty$. It can also be shown, by appealing to the integral
representation of the modified Bessel function, that $c(\nu)$ is
continuous with respect to $\nu$ when $\nu \geq 1/2$, from which we deduce
that $c^*: = \max_{[1/2, \nu^*]} c(\nu) <\infty\ .$ Now, for $1/2 \leq \nu
\leq \nu^*$ and $r \geq 4$ we have $\nu r \geq 2$ and
$$
 \exp(\nu r) \nu^{1/2} \vert K_\nu(\nu r)\vert
\ = \ r^{-1/2} \left[ \exp(\nu r) (\nu r)^{1/2} \vert K_\nu(\nu r)\vert \right] \  \leq\  \frac{1}{2} c(\nu) \
\leq  \ \frac{1}{2} c^*\ .
$$
This shows that the estimate \eqref{eq:estBessel} holds uniformly for all
$r \geq 4$ and $\nu \geq 1/2$. Using this in \eqref{eq:nowRHS} we then
have
\begin{align}\label{eq:nowRHS1}
\int_{\ell/\lambda}^\infty r^{d-1} \vert \kappa (r) \vert \rd r
\ \lesssim \   \frac{1}{{\Gamma(\nu)}} \left(\frac{\nu}{2}\right)^{\nu + d/2}
 \nu^{-1/2} \int_{ \sqrt{2/\nu}(\ell/\lambda)}^\infty
r^{d - 1+ \nu }  e ^{-\nu r} \rd r \ .
\end{align}

Now note that \eqref{eq:alest} implies that $\ell/\lambda \geq
C_2\nu^{1/2} $, and hence $\sqrt{2/\nu}(\ell/\lambda)\ge \sqrt{2}C_2$, so
we can choose $C_2$ a large enough constant independent of $\ell, h_0,
\lambda, \nu, \sigma^2$ so that $\sqrt{2/\nu}(\ell/\lambda) \geq 10$.
Then, using the elementary inequality $ x e \leq e^x $ which holds when $x
\geq 1$, it follows that $ (r/10)e \leq  \exp(r/10)$ when $r \geq 10$.
Hence for $r \geq  \sqrt{2/\nu}(\ell/\lambda) \geq 10$, we have
$$ r^{d-1+\nu} \ \leq \ \left( \frac{10}{e} \right)^{d-1+ \nu} \exp\left(
\frac{d-1+ \nu}{10} r \right) \ \leq  \left( \frac{10}{e} \right)^{d-1+ \nu}
e^{{\nu r}/{2}} \ ,
$$
where in the last step we used $d-1 \leq 2 \leq 4 \nu$. Inserting this
into the right-hand side of \eqref{eq:nowRHS1}, we get, after integration
and some manipulation,
$$\int_{\ell/\lambda}^\infty r^{d-1} \vert \kappa (r) \vert \rd r
\ \lesssim \   \frac{1}{{\Gamma(\nu)}} \left(\frac{5 \nu}{e}\right)^{\nu}
\nu^{d/2-3/2} \exp\left(- \sqrt{\frac{\nu}{2}} \frac{\ell}{\lambda}\right) \ .
$$
Stirling's formula implies that $\Gamma(\nu) \nu^{1/2} e^{\nu} \ \sim \
\nu^\nu $ and the estimate \eqref{eq:nowRHS2} follows.
\end{proof}

\begin{remark} \label{rem:Matern_extend}  The cases $\nu \geq  1/2$ in
  the \Mat family of covariances are of main interest in applications.
  In order to avoid further technicalities,  we have restricted our
  attention to those cases in the proof and used the lower bound on
  $\nu$ several  times in a non-trivial way. First of all it is used in
  the demonstration that estimates \eqref{eq:alphah0small},
  \eqref{eq:lhbelow} and \eqref{eq:nowRHS2} are sufficient to ensure the
  positive definiteness of $\Rext$. Then it is used in the proofs of
  each of the estimates \eqref{eq:lhbelow} and \eqref{eq:nowRHS2}. The
  extension of Theorem  \ref{cor:matern-growth} to $\nu \in (0,1/2)$
  remains an open question. 
\end{remark}

To complete the theory of this section, we discuss the case $\nu =
\infty$. In this case,
\begin{align*} %\label{eq:Gauss}
 \rho(\bsx) =
\kappa(\Vert\bsx \Vert_2/\lambda), \quad \text{with} \quad \kappa(r) =
\sigma^2 \exp(-r^2 / 2),
\end{align*}
and an elementary calculation gives
\begin{align*} %\label{eq:FTGauss}
\wrho(\bsxi) = \sigma^2 (2 \pi)^{d/2} \lambda^d \exp(- 2 \pi^2 \lambda^2
\Vert \bsxi\Vert_2^2).
\end{align*}
For this (Gaussian) covariance it is well-known that the Karhunen-Lo\'eve
expansion converges exponentially (see, e.g., \cite{ScTo:06}), in which
case it may be preferable to compute realisations of the field $Z$ via the
KL expansion, rather than the method proposed here. Nevertheless the
existing analysis can be applied to this case as we now show.
\begin{theorem}\label{thm:PDGauss}
For the Gaussian covariance (i.e., the \Mat kernel with $\nu = \infty$),
there exists a constant $B$ (depending only on spatial dimension $d$) such
that positive definiteness of $\Rext$ is guaranteed when
\begin{align*} %\label{eq:Gaussl}
 {\ell} \ \geq \ {1} \ + \ \lambda \max \left\{ \sqrt{2} \frac{\lambda}{h_0} , B\right\} \ .
\end{align*}
Hence, if $\lambda/h_0$ is fixed (i.e., a fixed number of grid points per
unit correlation length is used), then the $\ell$ required for positive
definiteness decreases linearly with decreasing correlation length
$\lambda$, until the minimal value $\ell = 1$ is reached.
\end{theorem}
\begin{proof}
As before, without loss of generality we set $\sigma^2 = 1$ and  we  make use of  Corollary \ref{cor:PD_isot}.
The left-hand side of \eqref{eq:corest} is then
$$ (2 \pi)^{d/2} \int_{3 \lambda d^{1/2}/ (2 h_0)}^\infty  \ r^{d-1} \exp(- 2 \pi^2 r ^2) \, \rd r \ = \
(2 \pi)^{-d/2} \int_{c_1 (\lambda/h_0)}^\infty  r^{d-1} \exp(-r^2/2)\,  \rd r   ,
$$
with $c_1 = 3 \pi \sqrt{d}> 1 $ . To make this quantity greater than the
right-hand side of  \eqref{eq:corest}, we require
$$ \int_{(\ell-h_0)/\lambda}^\infty  r^{d-1} \exp(-r^2/2)\,  \rd r
\ < \     c_2 \left( \frac{h_0}{\lambda} \right)^d \int_{c_1 \lambda/h_0}^\infty  r^{d-1} \exp(-r^2/2)\,  \rd r , $$
with $c_2 =   {2d(2 \pi d)^{-d/2} }{{(3^d-1)^{-1}\,3^{1-d}}}$.
We shall show that there exists $B$ (depending only on $d$) such that
\begin{align}\label{eq:assert}
 \int_{\ell/\lambda}^\infty  r^{d-1} \exp(-r^2/2) \, \rd r
 \ < \     c_2 \left( \frac{h_0} {\lambda} \right)^d \int_{c_1 \lambda/h_0}^\infty  r^{d-1} \exp(-r^2/2)\,  \rd r ,
\end{align}
when
\begin{align} \label{eq:Gaussl1}
 {\ell} \ \geq \ \lambda \max \left\{ \sqrt{2} c_1 \frac{\lambda}{h_0} , B\right\} \ ,
\end{align}
and the statement of the theorem then follows, since $h_0 \leq 1$.

To prove \eqref{eq:assert}, we estimate its left-hand side by
\begin{align}
 \int_{\ell/\lambda}^\infty  r^{d-1} \exp(-r^2/2) \, \rd r \
 &\ \leq \ \exp(-\ell^2/ 4 \lambda^2) \int_{\ell/\lambda}^\infty  r^{d-1} \,  \exp(-r^2/4) \, \rd r  \nonumber \\
 &\ = \ 2^{d/2} \exp(-\ell^2/ 4 \lambda^2) \int_{\ell/(\sqrt{2} \lambda)}^\infty  r^{d-1} \,  \exp(-r^2/2) \, \rd r \ .
\label{eq:Gaussl2}
\end{align}
We now choose $B$ (depending only on $d$) to have the property that
\begin{align} \label{eq:propB}
 \exp(-x^2/4) \, x^d \ \leq\  c_2 , \ \quad
 \text{when} \quad x \geq B .
\end{align}
Then  if $\ell$ satisfies \eqref{eq:Gaussl1}, we have
\begin{align}\label{eq:follow}
 \ell/\lambda \ \geq\  B\quad \text{ and}
 \quad \ell/\lambda \ \geq \  \sqrt{2} c_1 \lambda /h_0 >  \sqrt{2} \lambda / h_0 .
\end{align}
Thus, using \eqref{eq:propB} and \eqref{eq:follow}, we have
$$
 2^{d/2} \, \exp(-\ell^2 /4 \lambda^2)
 \ \leq\  c_2 \left( \sqrt{2}\lambda/\ell\right)^d
 \ < \  c_2 \left(h_0/\lambda\right)^d \ .
$$
Combining this with  \eqref{eq:Gaussl2} and using the fact that $\ell /
(\sqrt{2} \lambda) \geq c_1 \lambda/ h_0$ we obtain \eqref{eq:assert}.
\end{proof}

\section{Eigenvalue decay}
\label{sec:decay}

In this section we return to the formula \eqref{eq:Gauss_exp1} for
sampling the random field $\bsZ$. For the circulant embedding method,
using the notation introduced above, we have
\begin{equation*} %\label{eq:fse}
\bsZ(\omega) \ = \ \sum_{\bk \in \oZmd} \bsB_\bk Y_\bk(\omega) + \overline{\bsZ},
\end{equation*}
where $Y_\bk$ are i.i.d.\ standard normals, $\bsB_\bk$ are the columns of
the matrix $B$, and in this case $B$ is taken to be the appropriate $M$
rows of the matrix $\Bext$, as described in Theorem \ref{thm:decomp}.
Recalling \eqref{eq:defBext} and noting that every entry in $\Re(\calF) +
\Im(\calF)$ is bounded by $\sqrt{2/s}$, it follows that
 \begin{equation*} %\label{eq:bj-lam}
  \|\bsB_\bk\|_\infty \,\le\, \sqrt{\frac{2\,\Lambda^{\rm ext}_{\bk}}{s}}, \quad \bk \in \oZmd,
 \end{equation*}
where $\Lambda^{\rm ext}_{\bk}$ are the eigenvalues of the matrix $\Rext$.

In Quasi-Monte Carlo (QMC) convergence theory (see for example
\cite{GKNSSS:15}, \cite{paper2}) it is important to have good estimates
for $\|\bsB_\bk\|_\infty$. More precisely, arranging the $\Lambdaext_\bk$
in non-increasing order, it is important to study the rate of decay of the
resulting sequence. In order to obtain some insight into this question we
first recall  that the $\Lambdaext_\bk$ depend on both the regular mesh
diameter $h_0$ and the extension length $\ell$ or equivalently, on $h_0$
and $m$ (see \eqref{eq:defell}). Since $m = \len h_0^{-1}$, the dimension
$s$ given by \eqref{eq:defs} then grows if either $h_0$ decreases or
$\ell$ increases (or both). To indicate the dependence on these two
parameters, in this section we write variously
$$
\Lambdaext_\bk \ = \Lambdaext_\bk(h_0,\ell) \ = \Lambdaext_{s,\bk}.
$$

In order to get insight into the asymptotic behaviour of
$\Lambdaext_\bk(h_0,\ell)$, we study first the spectrum of the continuous
periodic covariance integral operator defined by
\begin{equation*} %\label{eq:cRext}
  {\cR}^{\rm ext}\, v (\bsx)
  \,:=\, \ \int_{[0,2\len]^d} \rho^{\rm ext} (\bsx-\bsx')\, v(\bsx') \,\rd\bsx' ,
  \quad \bsx \in [0, 2\len]^d \ ,
\end{equation*}
where $\rhoext$ is defined in \eqref{eq:rhotildef}.
This operator is a continuous analogue of the matrix $\Rext$ defined in \eqref{eq:Rext_def} and the
eigenvalues of each are closely related as we shall discuss below.

The operator $\cRext$ is a compact operator on the space of
$2\len$-periodic continuous functions on $\bbR^d$, and so it has a
discrete spectrum with the only accumulation point at the origin. Since
$\cRext$ is a periodic convolution operator, it is easily verified that
its eigenvalues and (normalised)  eigenfunctions (which depend on $\ell$)
are
\begin{align}\label{eq:perFT}
&\lambdaext_{\bsk}(\ell) \ = \ \int_{[0,2\ell]^d}
\rhoext( \bsx ) \exp\left(- {2 \pi \ri \, \bsxi_\bsk  \cdot \bsx}\right) \rd \bsx \
=\ \int_{[-\ell,\ell]^d}
\rho( \bsx ) \exp\left(- {2 \pi \ri \, \bsxi_\bsk  \cdot \bsx}\right) \rd \bsx\\
&\text{and} \quad v_\bk(\bsx) = (2\ell)^{-d/2} \exp(2 \pi \ri \,
\bsxi_\bsk \cdot \bsx), \quad  \text{and where} \quad
\bsxi_\bsk =  \bsk/ (2\ell)  \quad \text{and} \quad \bk \in \bbZ^d.\nonumber
\end{align}
Here we used the fact that $\rho^{\rm{ext}}$ is $2\ell$-periodic and that
$\rho^{\rm{ext}}$ and $\rho$ coincide on $[-\ell,\ell]^d$. The eigenvalues
 $ \lambdaext_\bk(\ell)$ are real but not necessarily positive, since
$\rhoext$, unlike $\rho$, may not be positive definite (for the definition
see \S\ref{subsec:posdef}).

The close relationship between the eigenvalues of the continuous operator
and the matrix eigenvalues is shown by the following theorem.

\begin{theorem} \label{thm:decay1}
For fixed $\ell\ge 1$ and fixed $\bk \in \overline{\bbZ}_m^d$, where
$m=\ell/h_0$, the matrix eigenvalues $\Lambdaext_{\bsk}(h_0,\ell)$,
weighted by $h_0^d$, converge to $\lambdaext_{\bsk}(\ell)$ as $h_0\to 0$:
$$h_0^d\,\Lambdaext_{\bsk}(h_0,\ell) \rightarrow \lambdaext_\bk(\ell),
\quad \text{as} \quad  h_0 \rightarrow 0. $$
\end{theorem}
\begin{proof}
The formula for $\Lambdaext_\bk(h_0,\ell) $ given by
\eqref{eq:eivsandeifs}, when weighted by $h_0^d$, can be seen as a
rectangle rule approximation of the second integral defining
$\lambdaext_\bk(\ell)$ in \eqref{eq:perFT}, with grid spacing of $h_0$.
Since the integrand in that integral is continuous, the convergence claim
holds.
\end{proof}

From now on  we restrict attention to the \Mat case (Example
\ref{ex:Mat}). Our first result shows that $\lambdaext_\bsk(\ell)$ is
exponentially close to the full range Fourier transform
$\wrho(\bsxi_\bsk)$ (given in \eqref{eq:FT}), and this holds true
uniformly in $\ell$ and $\bk$.

\begin{lemma}\label{thm:decay2}
For the case of the \Mat kernel, we have
\begin{equation}\label{eq:bound1}
\left\vert \lambdaext_\bsk(\ell) - \wrho(\bsxi_\bsk) \right\vert \ \lesssim \
\lambda^d  5^{\nu}
\nu^{d/2-1} \exp\left(- \sqrt{\frac{\nu}{2}} \frac{\ell}{\lambda}\right) \, , \quad
\text{for all} \quad \ell \geq 1 \quad \text{and} \quad \bk \in \overline{\bbZ}_m^d.
\end{equation}
Thus there exist positive constants $C_3, C_4$ (independent of $h_0, \len,
\lambda$ and  $\nu$), such that, for all $\eps >0$ and all $\bk \in
\oZmd$,
\begin{align} \label{eq:epsilon}\left\vert \lambdaext_\bsk(\ell) - \wrho(\bsxi_\bsk) \right\vert \ \leq \ \eps,
\quad \text{when} \quad   \ell/\lambda \ \geq \ C_3\,  \nu^{1/2}  + C_4\,  \nu^{-1/2}  \log(1/\eps).
 \end{align}
\end{lemma}

\begin{proof}
From \eqref{eq:perFT}, \eqref{eq:FT} and \eqref{defmatern},   it is easy
to see that
\begin{align*}
\left\vert \lambdaext_\bsk(\ell) - \wrho(\bsxi_\bsk) \right\vert
\ \leq \ \int_{\Vert \bsx \Vert_2 \geq \len} \vert \rho(\bsx)\vert  \,\rd \bsx
\sim \int_\len^\infty r^{d-1} \vert \kappa( r/\lambda) \vert \,\rd r
\ = \ \lambda^d \int_{\len/\lambda}^\infty r^{d-1} \kappa(r) \,\rd r ,
\end{align*}
with  $\kappa$ as in \eqref{defmatern}, and again we use Notation
\ref{not:tilde}. Then, using \eqref{eq:nowRHS2}, we obtain
\eqref{eq:bound1}.

Forcing the right-hand side of \eqref{eq:bound1} to be bounded above by
$\eps$, rearranging and taking logs, we obtain a sufficient condition of
the form
$$
 \sqrt{\frac{\nu}{2}} \, \frac{\len}{\lambda} \ \geq \ C_* +  \log(1/ \eps) +
 d \log \lambda  + \nu \log 5 + (d/2-1) \log \nu,
$$
for some parameter-independent constant $C_*$. Recalling that we assume
$\lambda \leq 1$, and using $\nu\ge 1/2$ and $\log(\nu)/\nu\lesssim
1$, the sufficiency condition in \eqref{eq:epsilon} follows.
\end{proof}

With Lemma \ref{thm:decay2} in mind, we now discuss the asymptotic
behaviour of $\wrho(\bsxi_{\bk})$, by making use of the analytic formula
\eqref{eq:Hankel1}. In order to define an appropriate ordering we make the
following definition.
\begin{definition}[{\em Ordering of the integer lattice}]  \label{def:ordering}
Since $\bbZ^d$ is countable we can write its elements as a sequence
$\{\bsk(j): j = 1,2, \ldots\}$, such that $\bsk(1) = \mathbf{0} \in
\bbZ^d$ and such that the sequence $\{\Vert \bsk(j) \Vert_2: j =
1,2,\ldots\}$ is non-decreasing. This ordering is not unique.
\end{definition}

\begin{theorem}\label{thm:decay3}
For the case of the \Mat kernel, we have
\begin{align} 0\  < \ \wrho(\bsxi_{\bsk(j)}) \
\lesssim \ \lambda^{-2\nu} \ (\nu \len^2)^{\nu + d/2} \,
 j^{-(1+2 \nu/d)}, \quad j = 1,2,\ldots\ .\label{eq:decay3}
\end{align}
\end{theorem}

\begin{proof}
By \eqref{eq:Hankel1}  and the definition $\bsxi_\bsk = \bsk/(2\ell)$, we
have
\begin{align}
0 \ < \ \wrho(\bsxi_\bsk) &\ \sim \ \lambda^d \nu^{\nu} \, \frac{\Gamma(\nu+d/2)}{\Gamma(\nu)} \,
\left(\nu +  \pi^2 \lambda^2\Vert \bsk\Vert_2^2 / (2 \len^2)\right)^{-(\nu + d/2)} .
\nonumber
\end{align}
Now by \eqref{eq:Gamma1} we have (for even  $d$), $ {\Gamma(\nu + d/2) } \
\lesssim \  \nu^{d/2} \,    {\Gamma(\nu)} $. Moreover the same estimate
also holds for $d$ odd, as can be seen by employing the first equation in
\eqref{eq:Gamma2}, and then the estimate $ \Gamma(\nu + 1/2) \ \lesssim \
\nu^{1/2} \Gamma(\nu)$ (from \cite[eq. (1.3)]{Ker}). Hence we have the
upper bound
\begin{align*}
 \wrho(\bsxi_\bsk) &\ \lesssim \ \lambda^d \nu^{\nu + d/2} \,
  \left(\nu + \pi^2 \lambda^2\Vert \bsk\Vert_2^2 / (2 \len^2)\right)^{-(\nu + d/2)}
  \ \lesssim  \ \lambda^d \left(\nu\ell^2\right)^{\nu + d/2}
  (\lambda \Vert \bk \Vert_2)^{-(2 \nu + d)} \\
  &\ = \ \lambda^{-2\nu}(\nu \ell^2)^{\nu+d/2} \|\bk\|_2^{-(2\nu+d)}.
\end{align*}

Now, to  complete the proof of \eqref{eq:decay3}, we shall show that
\begin{align}
 \Vert \bsk(j) \Vert_2 \ \sim \ j^{1/d} \ .
 \label{eq:order}
\end{align}
To obtain \eqref{eq:order}, for each $j$, define the set $T(j) := \{ \bsk
\in \bbZ^d : \Vert \bsk \Vert_2 \, \leq \, \Vert \bsk(j) \Vert _2 \}$.
Then $T(j)$ is a superset and a subset of two cubes:
\begin{align*} %\label{eq:inclusion}
 [-d^{-1/2} \Vert \bsk(j)\Vert_2, d^{-1/2} \Vert \bsk(j) \Vert_2]^d
 \ \subseteq \  T(j) \ \subseteq \ [-\Vert \bsk(j)\Vert_2, \Vert \bsk(j) \Vert_2]^d\ .
\end{align*}
(The right inclusion follows from the definition of $T(j)$. On the other
hand if $\bsk$ lies in the left-most cube, then $\Vert \bsk \Vert_\infty
\leq d^{-1/2} \Vert \bsk(j) \Vert_2$ which implies $\Vert \bsk \Vert_2
\leq \Vert \bsk(j) \Vert_2$, i.e., $\bsk \in T(j)$.) Then it follows that
$\# T(j) \,  \sim \,  \Vert \bsk(j)  \Vert_2 ^d$. But also by definition
$T(j)$ contains $j-1$ elements of $\bbZ^d$. Hence \eqref{eq:order}
follows.
\end{proof}

Combining Lemma~\ref{thm:decay2} and Theorem~\ref{thm:decay3} with Theorem
\ref{cor:matern-growth}, we obtain the following  criterion which
simultaneously guarantees the positivity of all the eigenvalues of $\Rext$
and provides an  upper bound for $\lambda_{\bk(j)}(\ell)$, explicit in the
parameters $h_0, \ell, \lambda$ and $\nu$.
\begin{corollary} \label{cor:decay4} For the \Mat kernel,
\text{suppose } \quad
$$
 \ell/\lambda \ \geq \ \max\{ C_1,  C_3 \nu^{1/2}\} \,    +
 \max\{C_2 \nu^{1/2} ,  C_4 \nu^{-1/2}\}  \log\left(\max\{ \lambda/h_0, \nu^{1/2}\}\right).
$$
Then  $\Lambdaext_{\bk}(h_0,\ell) \geq 0$ for all $\bk \in \oZmd$ and also
\begin{align} \label{eq:decaylast}
 \lambdaext _{\bk(j)}(\ell) \ \leq \ \min\{ h_0/\lambda, \nu^{-1/2} \} \ + \
 \lambda^{-2\nu} \ (\nu \len^2)^{\nu + d/2} \,  j^{-(1+2 \nu/d)}, \quad j = 1,2,\ldots\
\end{align}
\end{corollary}

\begin{discussion} \label{conj:decay}
In the paper \cite{paper2} the authors of this paper analyse the
convergence of Quasi-Monte Carlo methods for uncertainty quantification of
certain PDEs with random coefficients, realisations of which are computed
using the circulant embedding technique described here. A
dimension-independent convergence estimate for the QMC method is proved
there under the assumption that given $d$ and covariance function $\rho$,
there exist $p<1$ and $C>0$ such that for all integers $m_0$, with $m$
chosen as in Algorithm~\ref{alg1} and $s=(2m)^d$, we have
\begin{equation} \label{eq:QMCcriterion}
 \sum_{\bk\in {\overline{\bbZ}_m^d}}
 \bigg(\frac{\Lambda^{\rm ext}_{s,\bk}}{s} \bigg)^{p/2}  \ \leq \ C\ .
\end{equation}
If $\Lambdaext_{s,j}$ for $j=1,\ldots,s$ denotes the eigenvalues
$\Lambdaext_{s,\bk}$ reordered so as to be nonincreasing, then a
sufficient condition for this to hold is that for some $\beta>1$ and
$C>0$, independent of $j$ and $s$, there holds
\begin{equation}\label{eq:decay_beta}
\sqrt{\frac{\Lambda^{\rm ext}_{s,j}}{s}} \le C j^{-\beta}\quad \text{for} \quad j=1,\ldots,s,
\end{equation}
in which case \eqref{eq:QMCcriterion} holds for $p$ any number in the open
interval $(1/\beta,1)$. Based on the known rate of decay of the analogous
{Karhunen-Lo\`eve} eigenvalues (e.g., \cite{LPS14}, \cite{GKNSSS:15}) and
the fact that  the second term in \eqref{eq:decaylast} decays with the
same rate (and bearing in mind the result of  Theorem \ref{thm:decay1}),
we conjecture that
\begin{align} \label{eq:conjecture}
 \mbox{\textbf{Conjecture}:\quad For the \Mat kernel, \ \eqref{eq:decay_beta} holds with}\quad
 \beta = \frac{1 + 2\nu/d}{2}\ .
\end{align}
Numerical experiments in the next section support this conjecture. Thus we
need to assume that $\nu > d/2$ in order to ensure that $\beta>1$ and
$p<1$. (If $\nu\leq  d/2$ then there is no predicted convergence rate for
the QMC algorithm. A similar theoretical barrier was detected in
\cite{GKNSSS:15}).

The conjecture that $\sqrt{\Lambda^ {\rm ext}_{s,j}/s}\le C
j^{-(1+2\nu/d)/2}$ remains a conjecture in spite of the results in Lemma
\ref{thm:decay2} and Theorem \ref{thm:decay3}. That is because these
results concern the continuous eigenvalues $\lambdaext_{\bk}(\ell)$ and
the result  which connects these with the eigenvalues $\Lambdaext_{\bk}
(h_0,\ell)$ (namely Theorem \ref{thm:decay1}) does not provide enough
information about the small eigenvalues. Indeed it is well known, and
demonstrated in the numerical experiments in the next section, that the
small matrix eigenvalues depart significantly from the corresponding
eigenvalues $\lambda^{\rm{ext}}_{\bk}(\ell)$ of the integral operator.

A sufficient condition for \eqref{eq:decay_beta}, and hence for
\eqref{eq:QMCcriterion}, is that \eqref{eq:decay_beta} holds not for all
eigenvalues  but for some fixed fraction of the eigenvalues, say $\alpha$,
starting from the largest. Suppose, for example, that for some fixed
$\alpha \in (0,1)$ we have
\[
{\sqrt{\frac{\Lambda^{\rm ext}_{s,j}}{s}}}
\le C j^{-\beta}\quad \text{for} \quad j=1,\ldots,\lceil \alpha s \rceil.
\]
Then because the eigenvalues are ordered, we also have, for $j=\lceil
\alpha s \rceil + 1,\ldots, s$,
\[
{\sqrt{\frac{\Lambda^{\rm ext}_{s,j}}{s}}}
\le C \lceil \alpha s \rceil^{-\beta}\
= C j^{-\beta} \left(\frac{j}{ \lceil \alpha s \rceil}\right)^\beta
\le \frac{C}{\alpha^\beta}j^{-\beta},
\]
so that the sufficiency condition \eqref{eq:decay_beta} is satisfied for
all $j = 1, \dots, s$ with an appropriately larger value of the constant.
\end{discussion}

\section{Numerical Experiments}
\label{sec:Numerical}

In this section we perform numerical experiments illustrating the
theoretical results given above, for the \Mat covariance in 2D and 3D. In
all experiments we set the variance  $\sigma^2 = 1$. In
\S\ref{subsec:posdefexpts}, we illustrate the positive definiteness
results of Theorem \ref{cor:matern-growth}. Then, in \S\ref{subsec:decay},
we illustrate the actual decay of the eigenvalues $\Lambdaext_{s,j}$ and
compare it to \eqref{eq:conjecture}.

\subsection{Positive definiteness}
\label{subsec:posdefexpts}

We investigate experimentally the minimal value of $\ell$ needed to ensure
that the extended circulant matrix $\Rext$ is positive definite.  Recall
that Theorem \ref{positivityoftransf}  guarantees the existence of such an
$\ell$ and Theorems \ref{cor:matern-growth} and \ref{thm:PDGauss} quantify
its behaviour in the \Mat case.

The plots in Figure~\ref{fig:1} show the behaviour of $\ell$ as a function
of $\log h_0^{-1} = \log m_0$ for various choices of $d$, $\nu< \infty$
and $\lambda$. They clearly show that $\ell$ depends linearly  on $\log
m_0$,  for $m_0$ large enough. They also indicate that $\ell$ gets smaller
if either $\lambda$ or $\nu$ gets smaller, all in accordance with
Theorem~\ref{cor:matern-growth}. For fixed $d$, $h_0$ and $\lambda$,
Theorem~\ref{cor:matern-growth} gives a lower bound  on $\ell$ which grows
like $\nu^{1/2}$. We observe this behaviour clearly for $d=2$, but the
growth with respect to $\nu$ is a bit slower for  $d=3$. (See
Figure~\ref{fig:2}, which  plots $\log \ell$ against $\log \nu$.) The
small triangle embedded in each graph indicates $\mathcal{O}(\nu^{1/2})$
growth.

\begin{table}%[th]
\centering
\begin{tabular}{|l|l|l|l|}
\hline
$\lambda$ & $m_0$ & $\ell$ ($d=2$) & $\ell$ ($d=3$)\\
\hline $1$ & $8$ & $8$ & $9$\\
\hline $0.5$ & $16$ & $4$ & $4.5$\\
\hline $0.25$ & $32$ & $2$ & $2.25$\\
\hline $0.125$ & $64$ & $1$ & $1.125$ \\
\hline
\end{tabular}
\bigskip
\caption{Illustration of Theorem \ref{thm:PDGauss}: Values of $\ell$
required to ensure positive definiteness of $\Rext$ in the case $\nu =
\infty$, with $\lambda$ decreasing and $m_0 \lambda = 8$ fixed.
}\label{tab:1}
\end{table}

Table \ref{tab:1} illustrates the result of Theorem \ref{thm:PDGauss}
which covers the case $\nu = \infty$. Here we tabulate the value of $\ell$
needed to ensure positive definiteness of $\Rext$, when decreasing
$\lambda$ and keeping $m_0\lambda$ fixed at $8$. Theorem \ref{thm:PDGauss}
gives a bound which decreases linearly in $\lambda$, until the minimum
$\ell = 1$ is reached. This behaviour is exactly as observed in Table
\ref{tab:1}. In this case $\Rext$ has many very small eigenvalues and we
deem it to be positive definite when all eigenvalues are positive,
ignoring those eigenvalues which are less than $10^{-13}$ in modulus.

\begin{figure}%[ht!]
\centering
\begin{tabular}{l} % I left the tabular to align the graphs to the left, otherwise the second row is slightly off
  \hspace*{-4mm}
  \includegraphics[scale=0.63,trim=0 0 0 0,clip]{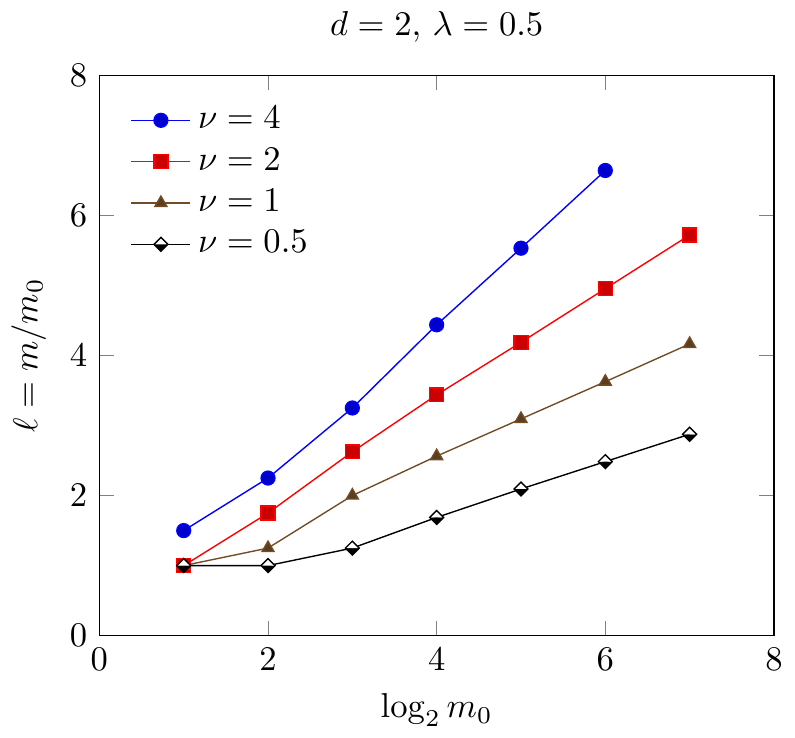}
  \includegraphics[scale=0.63,trim=14bp 0 0 0,clip]{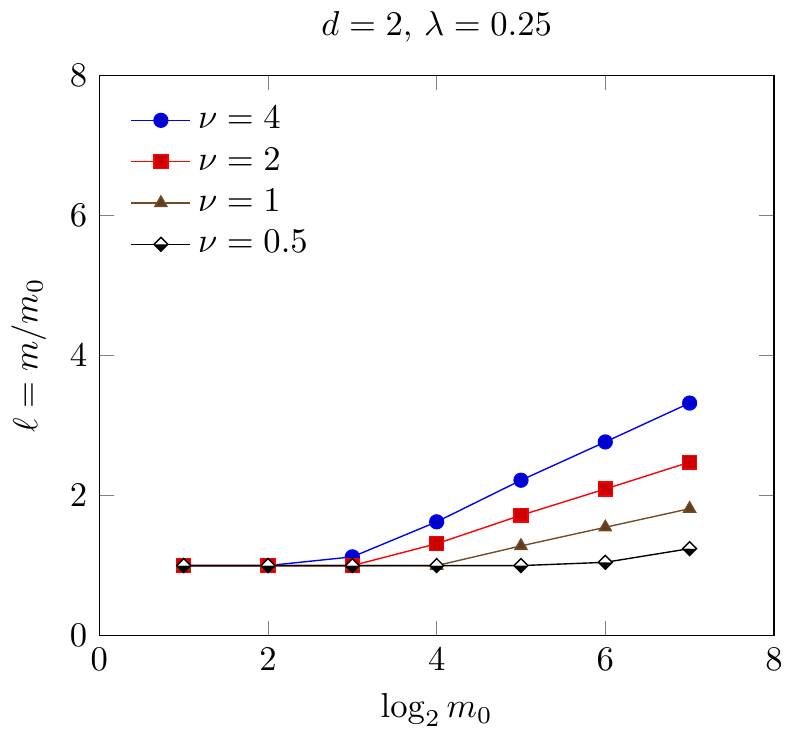}
  \includegraphics[scale=0.63,trim=14bp 0 0 0,clip]{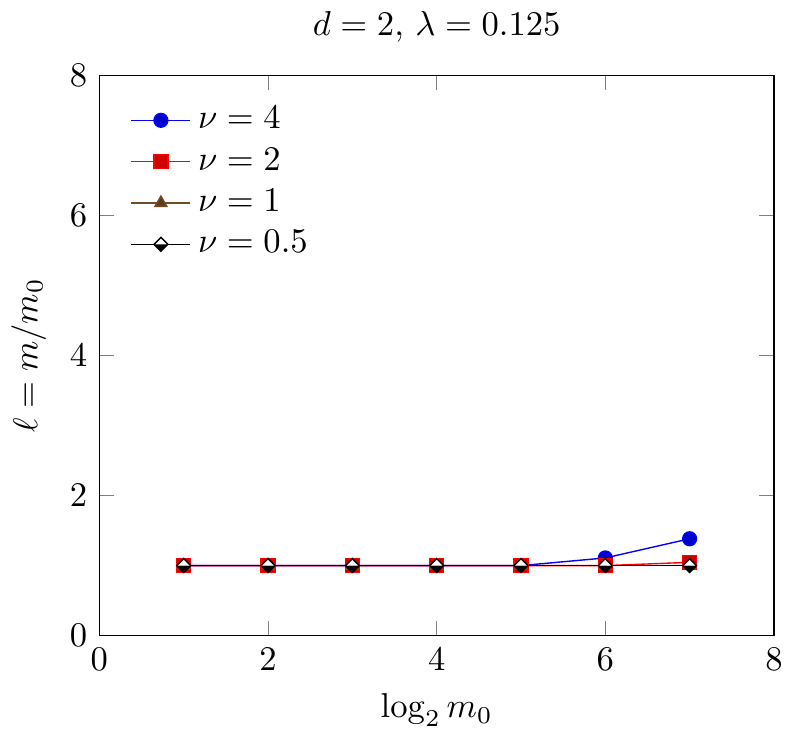} \\[2mm]
  \hspace*{-4mm}
  \includegraphics[scale=0.63,trim=0 0 0 0,clip]{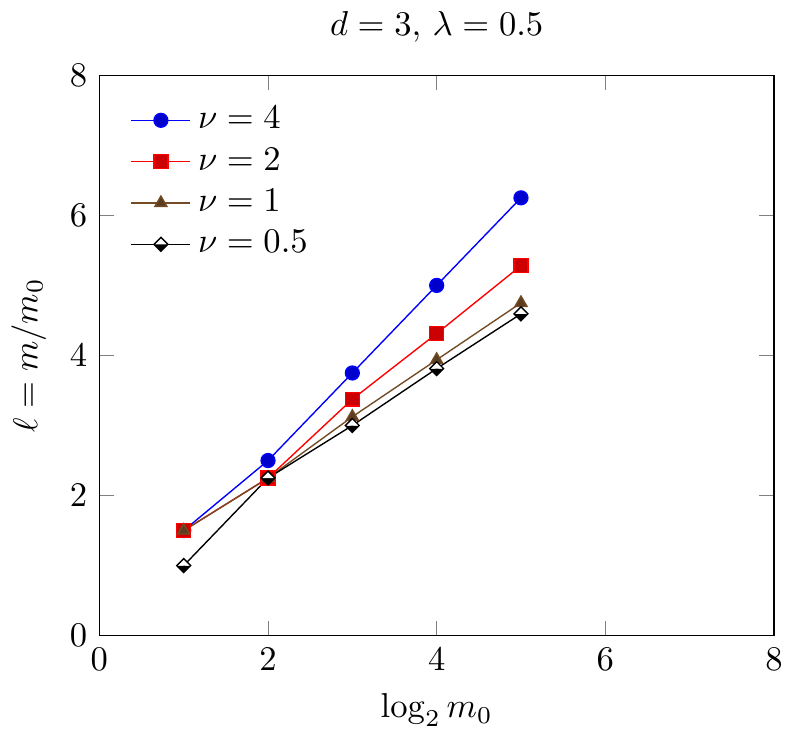}
  \includegraphics[scale=0.63,trim=14bp 0 0 0,clip]{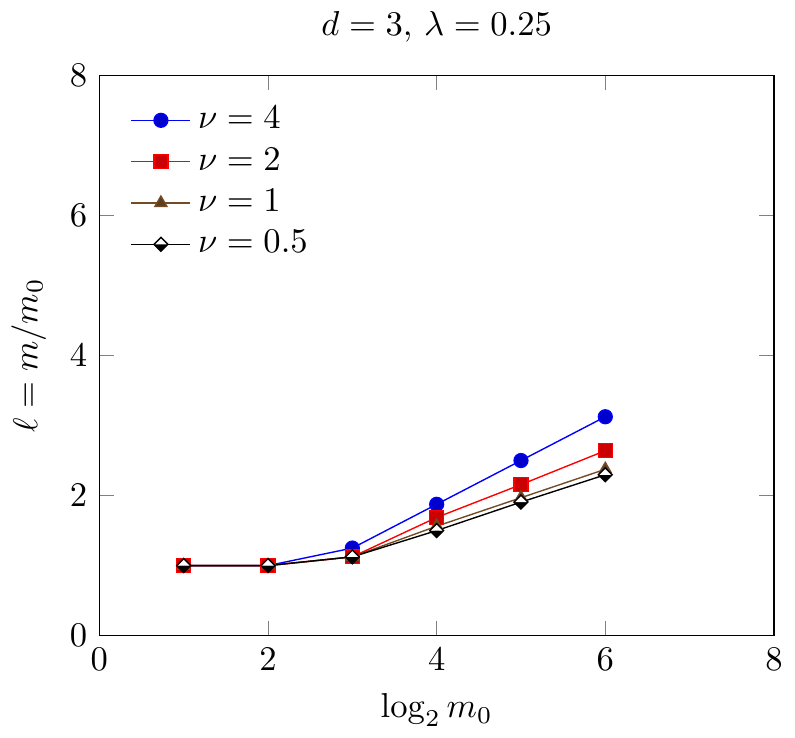}
  \includegraphics[scale=0.63,trim=14bp 0 0 0,clip]{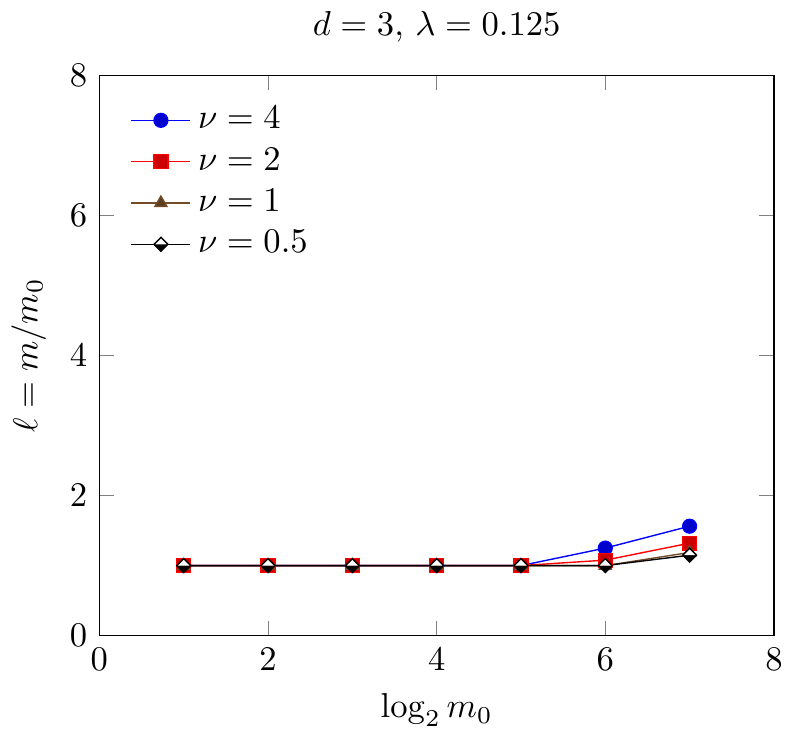}
\end{tabular}
\caption{Illustration of Theorem \ref{cor:matern-growth}: Graphs of the minimum {needed} value of $\ell$ {to obtain positive definiteness} against {$\log_2 h_0^{-1} = \log_2 m_0$} for different choices of {$d$, $\lambda$ and $\nu$}.
{The graphs show a linear relationship with lower values of $\ell$ for smaller $\lambda$ and for smaller $\nu$.}
\label{fig:1}}
\end{figure}

\begin{figure}%[p!]
\centering
\begin{tabular}{l} % I left the tabular to align the graphs to the left, otherwise the second row is slightly off
  \hspace*{-3mm}
  \includegraphics[scale=0.63,trim=0 0 0 0,clip]{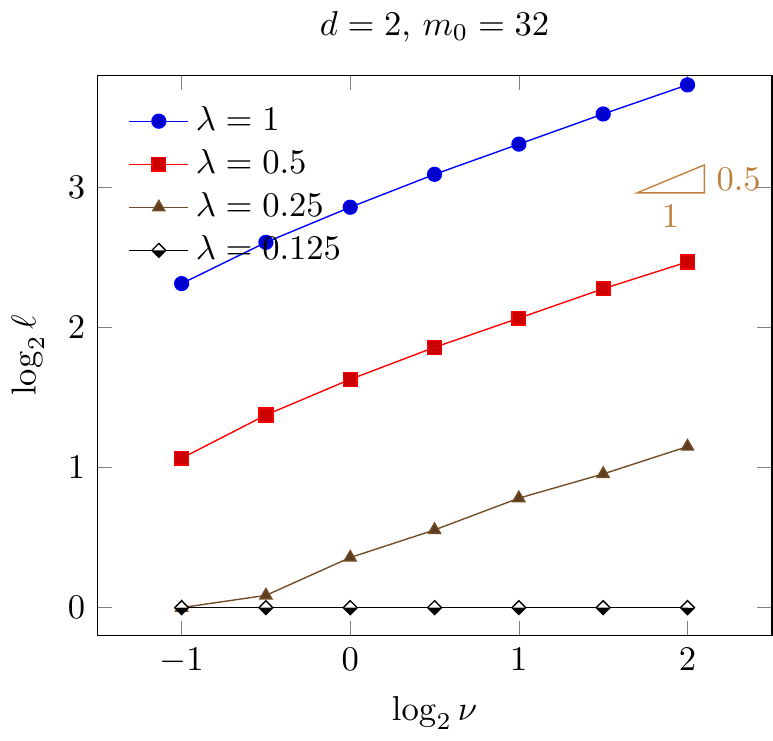}
  \includegraphics[scale=0.63,trim=14bp 0 0 0,clip]{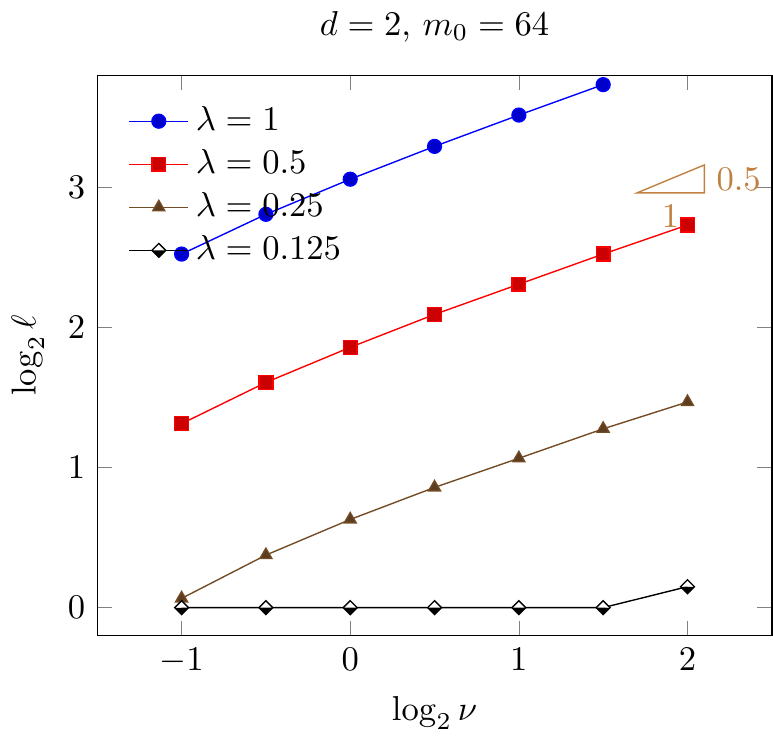}
  \includegraphics[scale=0.63,trim=14bp 0 0 0,clip]{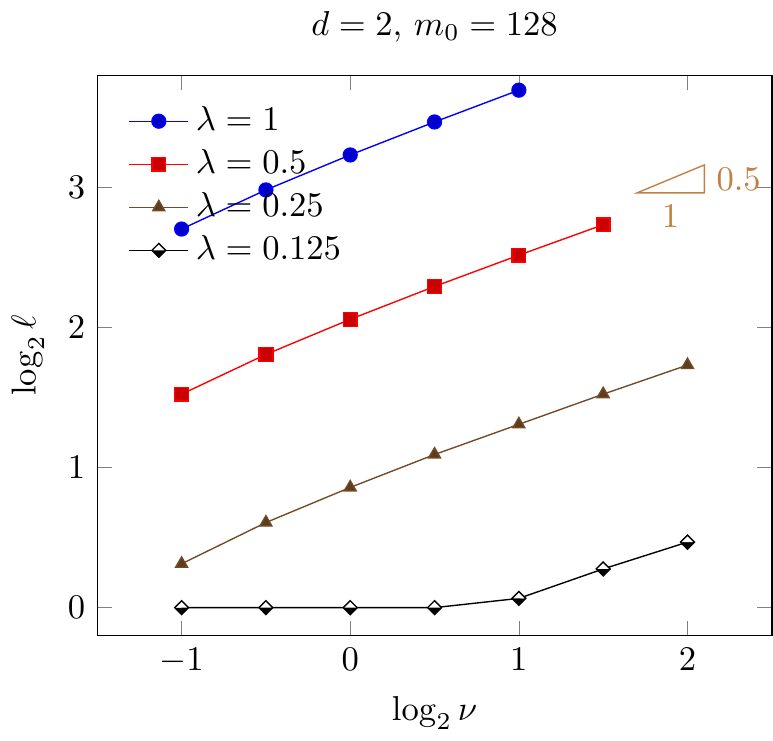} \\[2mm]
  \hspace*{-3mm}
  \includegraphics[scale=0.63,trim=0 0 0 0,clip]{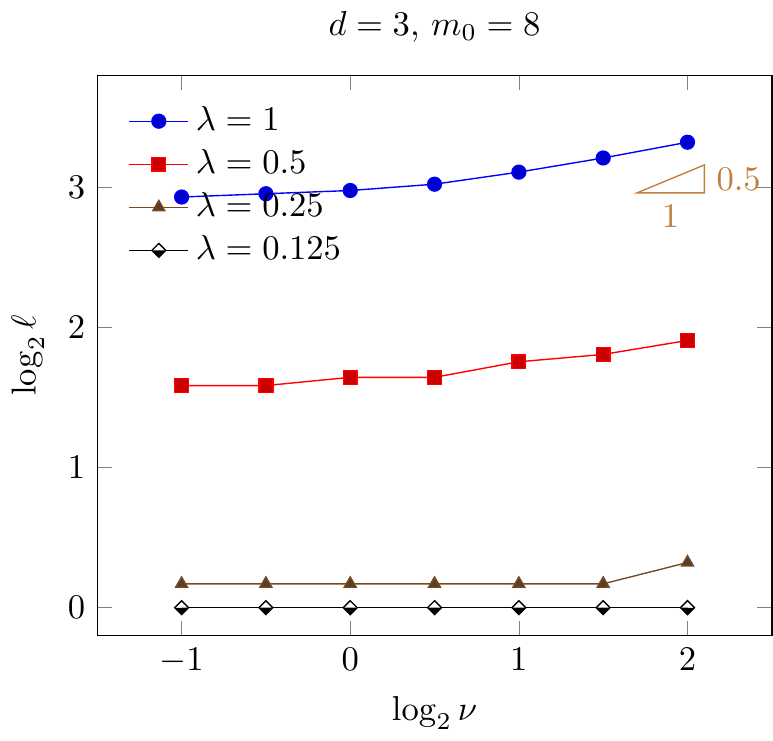}
  \includegraphics[scale=0.63,trim=14bp 0 0 0,clip]{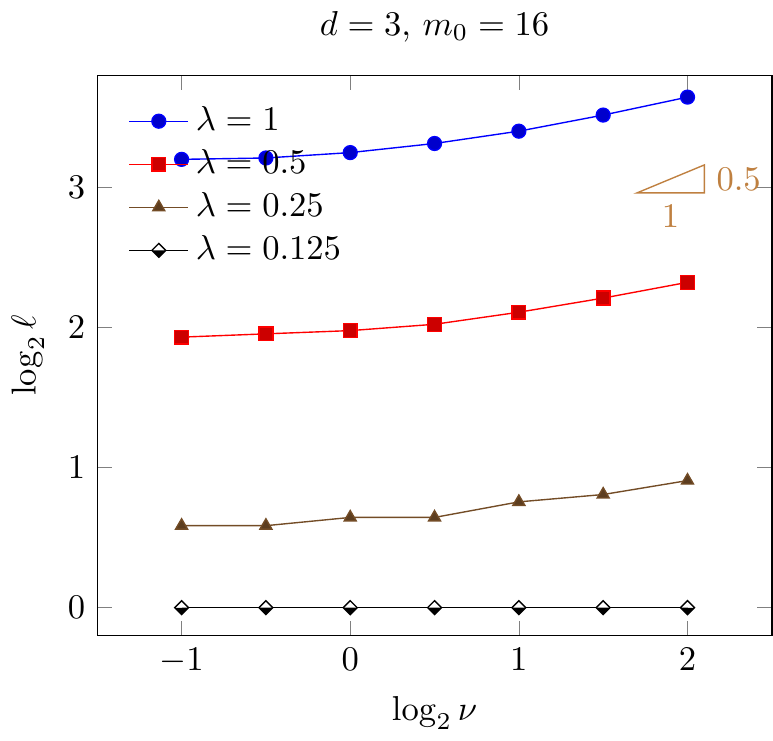}
  \includegraphics[scale=0.63,trim=14bp 0 0 0,clip]{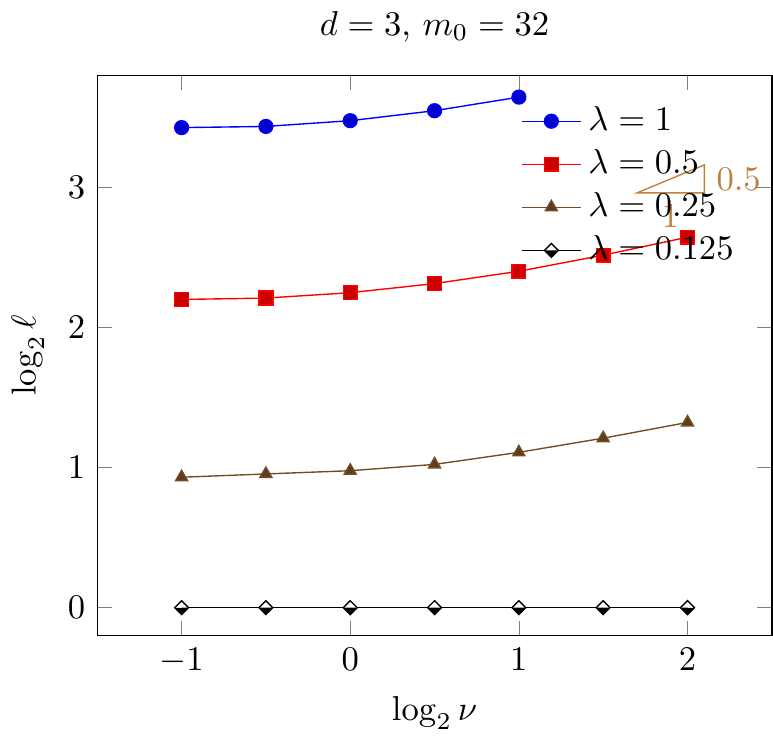}
\end{tabular}
\caption{Illustration of Theorem \ref{cor:matern-growth}:
{Log-log graphs} of the minimum {needed} value of $\ell$ {to obtain positive definiteness} against $\nu$
for various choices of {$d$, $m_0$ and $\lambda$}. {The small
triangle depicts a gradient of $0.5$}. \label{fig:2}}
\end{figure}

\subsection{Eigenvalue decay}
\label{subsec:decay}

Here we perform experiments to verify the decay conjecture
\eqref{eq:conjecture}. In Figures~\ref{fig:3a} and~\ref{fig:3b} we present
log-log plots of $\sqrt{\Lambdaext_{s,j}/s}$ against the index~$j$ for
various choices of $d$, $\nu$ and $\lambda$. For decreasing $h_0$ (i.e.,
increasing $m_0$) we determine the minimum $\ell$ to achieve positive
definiteness and make the plot for each case. If the conjecture
\eqref{eq:conjecture} holds, then we expect to see a polynomial decay of
rate $-(1 + 2\nu / d)/2$, as shown in the slope triangle in each case.
(The convergence may tail off eventually; note that Theorem
\ref{thm:decay1} does not provide an explicit convergence rate.) We see
that the computations closely follow the prediction
of~\eqref{eq:conjecture}. Moreover, to have the constant $C$
in~\eqref{eq:QMCcriterion} for the minimal embedding to be absolutely
bounded independent of $h_0$ (and thus also independent of $s$, $m$ and
$\ell$) we observe the lines to get closer together for increasing $m_0$.

Thus, while the numerical evidence supports conjecture
\eqref{eq:conjecture}, it remains an interesting open problem to prove
it.

\begin{figure}%[p!]
  \centering
  \includegraphics{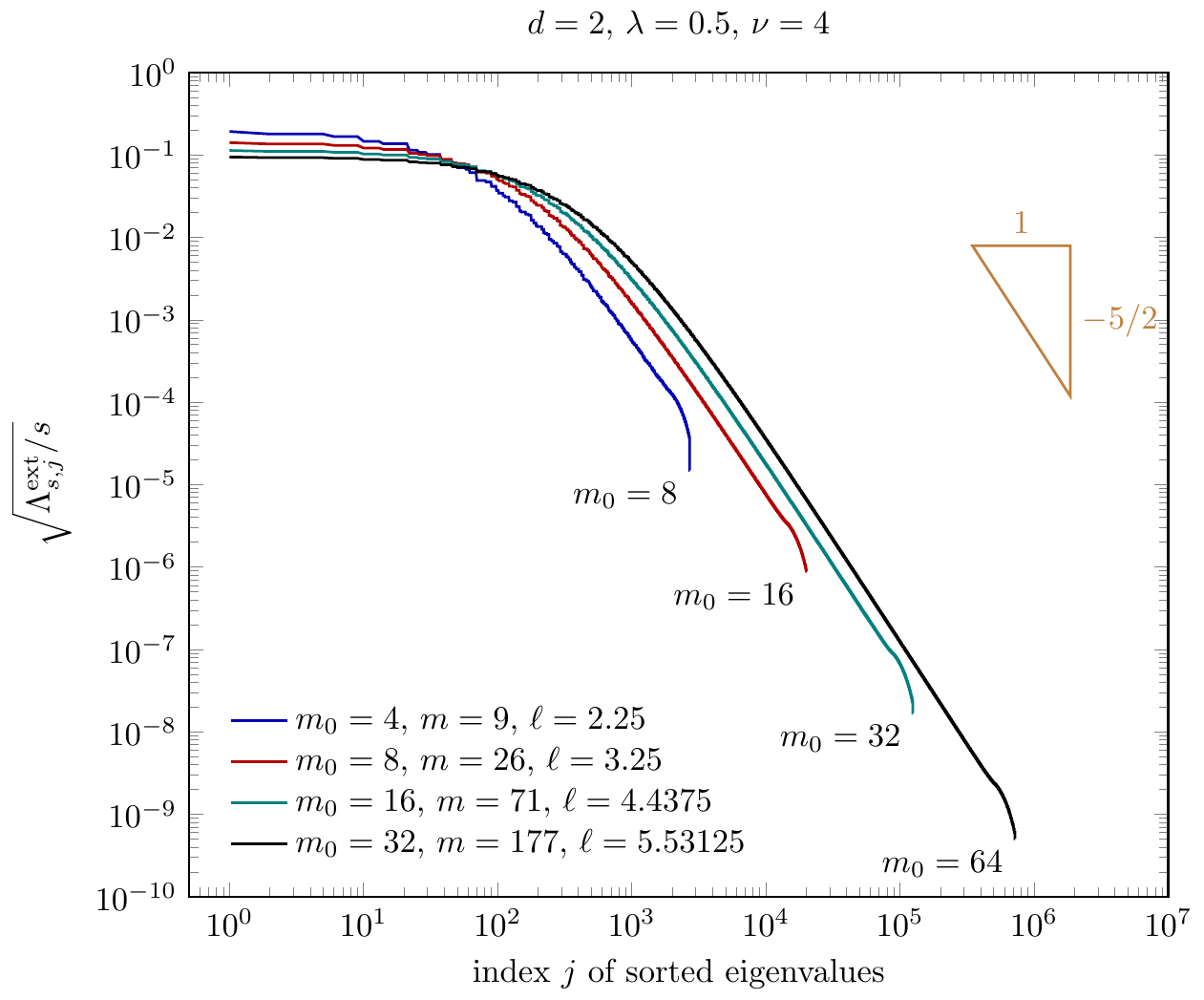} \\[2mm]
  \includegraphics{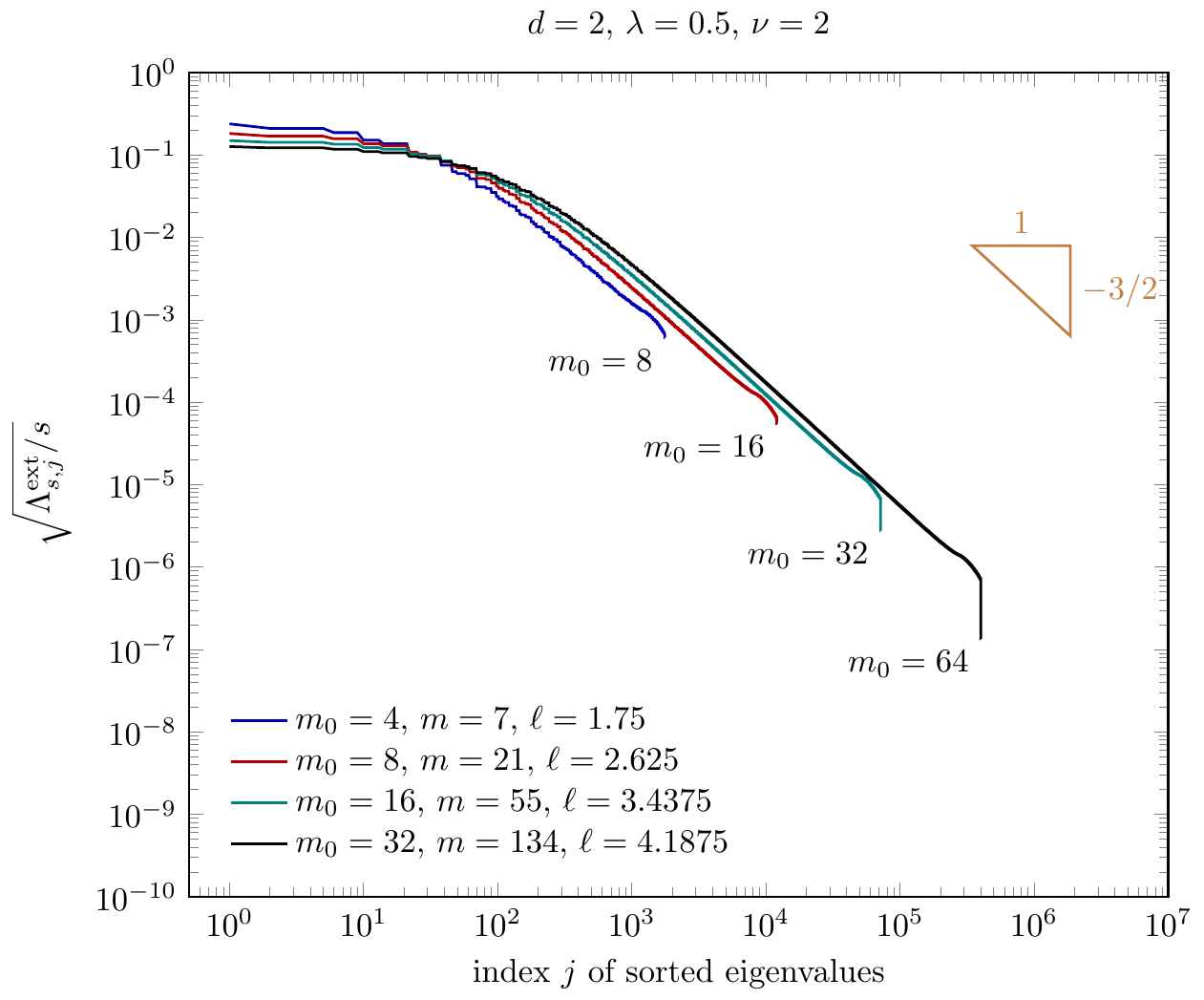}
  \caption{{Loglog plot of the decay of the eigenvalues in case of the minimal embedding
  for $d=2$, $\lambda=0.5$, and $\nu=4$ (top) and $\nu=2$ (bottom).
  The expected decay rates~\eqref{eq:conjecture} are marked by the slope triangles.}}\label{fig:3a}
\end{figure}

\begin{figure}%[p!]
  \centering
  \includegraphics{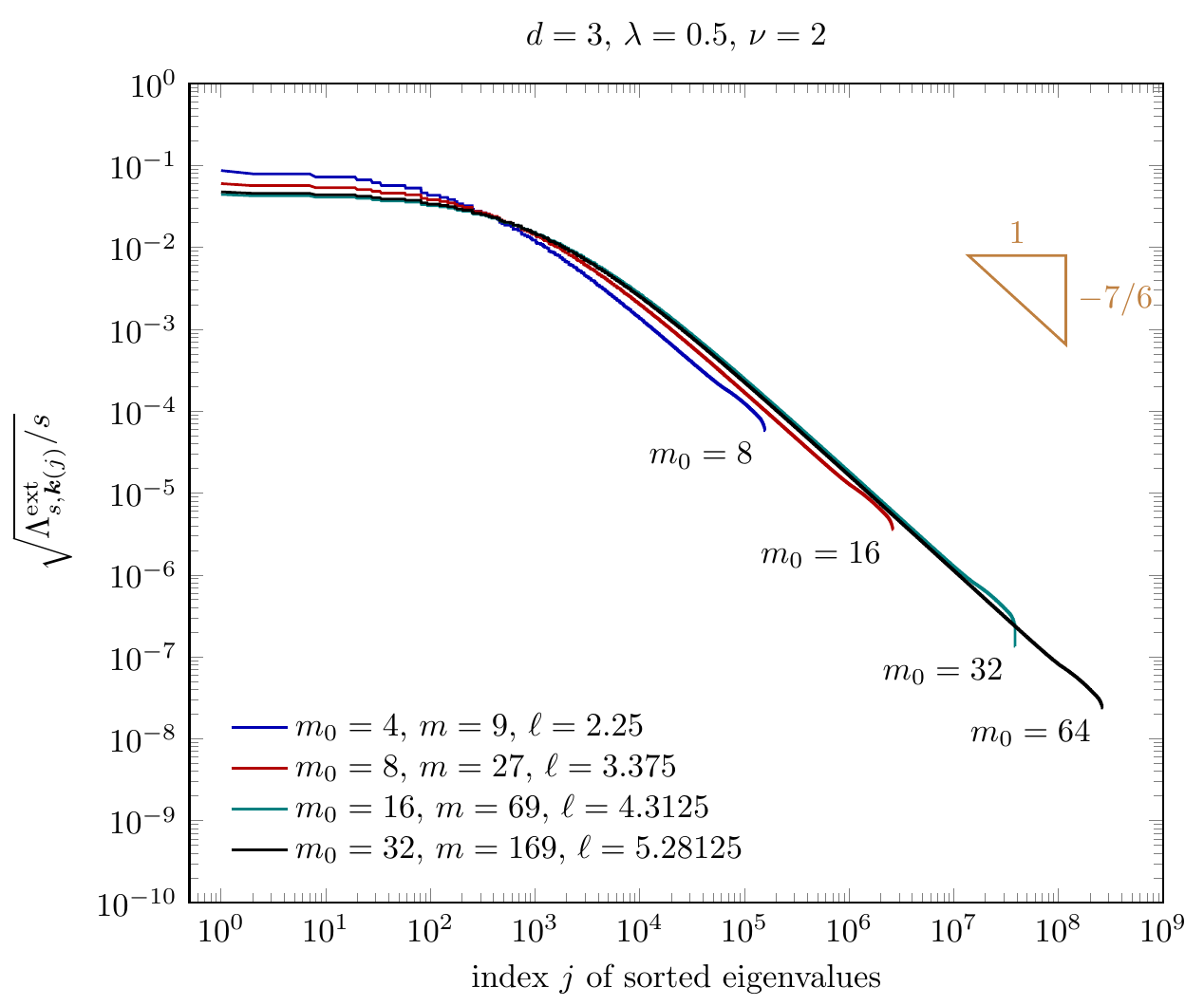}
  \caption{{Loglog plot of the decay of the eigenvalues in case of the minimal embedding
  for $d=3$, $\lambda=0.5$ and $\nu=2$.
  The expected decay rate~\eqref{eq:conjecture} is marked by the slope triangle.}}\label{fig:3b}
\end{figure}

\appendix
\section{Sampling theorem}

The following result is a $d$-dimensional version of what is commonly
called a sampling theorem.  (See \cite{Champ} for a $1$-dimensional
version.)

\begin{theorem}\label{thm:sampling}
Let $\rho \in L^1(\mathbb{R}^d)$ be real-valued and symmetric, with
$\widehat{\rho} \in L^1(\mathbb{R}^d)$. Suppose also that for some
${h}>0$,
\begin{equation}\label{eq:doubles}
 \sum_{\bk \in \bbZ^d} |\rho(h \bk)|<\infty.
\end{equation}
 Then
\begin{align}\label{eq:samplings}
 \sum_{\bk \in \bbZ^d } \rho({h} \bk )\,\exp(-2\pi\ri {h} \bk\cdot\bsxi )
 &\,=\, \frac{1}{{h}^d}\sum_{\br \in \bbZ^d } \widehat\rho\left(\bsxi + \frac{\br}{{h}} \right)
 \qquad \text{for {almost} all }\bsxi \in \mathbb{R}^d .
\end{align}
If $\widehat{\rho}$ is everywhere positive, then for all $\bsxi \in
\mathbb{R}^d$
\begin{align}
 \sum_{\bk\in\bbZ^d}
 \rho(h \bk )\,\exp(-2\pi\ri h \bk\cdot\bsxi )
\,& \ge\, \mathrm{essinf}_{\bszeta\in[-\frac{1}{2h},\frac{1}{2h}]^d}
 \frac{1}{h^d}\sum_{\br \in \bbZ^d }
 \widehat\rho\left(\bszeta + \frac{\br}{h} \right) \nonumber \\
 \,& \ge \,
 {\max_{\br\in \bbZ^d}}\min_{\bszeta\in[-\frac{1}{2},\frac{1}{2}]^d}
 \frac{1}{h^d}
 \widehat\rho\left(\frac{\bszeta+\br}{h}\right)  > 0.
\label{eq:min-bound}
\end{align}
\end{theorem}

\begin{proof}
Note first that $\rho$ and $\wrho$ are both continuous because of the
assumed integrability of $\wrho$ and $\rho$ respectively.  Moreover
$\wrho$ is real because of the assumed symmetry of $\rho$. The infinite
sum on the left-hand side of \eqref{eq:samplings} is a well-defined
continuous function of $\bsxi \in \bR^d$, because the convergence is
absolute and uniform by virtue of \eqref{eq:doubles}.

Now consider the right-hand side of \eqref{eq:samplings}, writing it for
convenience as
\[
 g_h(\bsxi) \,:=\,
 \frac{1}{h^d}\sum_{\br \in \bbZ^d} \widehat \rho\left( \bsxi + \frac{\br}{h} \right),
 \quad \bsxi \in \bR^d .
\]
We will show that the function so defined is locally integrable. It is
also manifestly $1/h$ periodic in each coordinate direction, since
$g_h(\bsxi+\bsq/h) = g_h(\bsxi)$ for all $\bsq\in\bbZ^d$.

To show the local integrability of $g_h$ we will show that the sum
defining $g_h$ converges absolutely to an integrable function on
$[0,1/h]^d$. Letting $\Xi$ denote any finite subset of $\mathbb{Z}^d$, we
have
\begin{align*}
 \int_{[0,1/h]^d}\Bigg|\sum_{\br\in\Xi}\widehat{\rho}\left(\bsxi + \frac{\br}{h}\right)\Bigg|
 \,  \rd \bsxi \,
 &\le\,\int_{[0,1/h]^d}\sum_{\br\in\Xi}
 \Bigg|\widehat{\rho}\left(\bsxi + \frac{\br}{h}\right)\Bigg|\, \rd\bsxi
 \,=\, \sum_{\br \in\Xi}\int_{[0,1/h]^d}\left|\widehat{\rho}\left(\bsxi + \frac{\br}{h}\right)\right|
  \,\rd \bsxi \\
 &\,=\,\sum_{\br \in\Xi} \int_{r_1/h}^{(1+r_1)/h} \cdots \int_{r_d/h}^{(1+r_d)/h}|\widehat{\rho}(\bsxi)|
 \,\rd \bsxi\,
 \,\le\, \int_{\mathbb{R}^d}|\widehat{\rho}(\bsxi)|\,\rd \bsxi
 \,<\, \infty,
\end{align*}
where the last step follows from the assumed integrability of $\wrho$. It
then follows from the dominated convergence theorem that $g_h$ is
integrable on $[0,1/h]^d$.

We now show that the left-hand side of \eqref{eq:samplings} is just the
Fourier series of the integrable $(1/h)$-periodic function $g_h$. The
Fourier coefficients of $g_h$ are given by
\[
  \widehat{g_h}(\bk)
  = h^d \int_{-1/(2h)}^{1/(2h)} \cdots \int_{-1/(2h)}^{1/(2h)}
  \left(\frac{1}{h^d}\sum_{\br \in \bbZ^d}
  \widehat \rho\left(\bsxi + \frac{\br}{h}\right)\right)\exp(2\pi\ri h\bsk\cdot\bsxi )\,\rd \bsxi ,
\]
where the unconventional sign in the exponent is valid because $g_h$ is
real and symmetric. The order of integration and summation can be
interchanged by a second application of the dominated convergence
argument, to give
\begin{align*}
 \widehat{g_h}(\bk )
 \, & =\,\sum_{\br \in \bbZ^d} \int_{-1/(2h)}^{1/(2h)} \cdots \int_{-1/(2h)}^{1/(2h)}
 \widehat \rho\left(\bsxi + \frac{\br}{h}\right)
 \exp(2\pi\ri h \bk\cdot\bsxi )\, \rd \bsxi\\
 &=\,\sum_{\br \in \bbZ^d}
 \int_{(-1/2 + r_1)/h}^{(1/2 + r_1)/h}{\cdots\int_{(-1/2 + r_d)/h}^{(1/2 + r_d)/h}}
 \widehat{\rho}(\bsxi)\exp(2\pi\ri h \bk\cdot\bsxi )\,  \rd \bsxi\\
 &=\,\int_{\mathbb{R}^d}\widehat{\rho}(\bsxi)\exp(2\pi\ri h \bk\cdot\bsxi )\,\rd \bsxi
 \, = \, \rho(h \bk),
\end{align*}
where in the last step we used the fact that the inverse Fourier transform
of $\widehat{\rho}$ recovers $\rho$ because of the integrability of
$\rho$. Thus the left-hand side of \eqref{eq:samplings} is just the
Fourier series of $g_h$, as asserted. From this it follows that the
integrable function $g_h$ is equal almost everywhere to to the continuous
function on the left-hand side, so completing the proof of
\eqref{eq:samplings}.

If $\widehat{\rho}$ is a positive function it follows that the left-hand
side of \eqref{eq:samplings} is bounded below by the essential infimum of
$g_h(\bsxi)$ over $\bsxi$, which in turn can be bounded below by retaining
only the largest term in the sum, which is positive.
\end{proof}

\end{document}